\documentclass[12pt]{amsart}

\usepackage[left=1in, right=1in, top=.75in, bottom=.75in]{geometry} 
\usepackage[utf8]{inputenc}
\usepackage{amssymb, enumerate}
\usepackage{amsmath, eqnarray, amsthm}
\usepackage{graphicx} 
\usepackage{wrapfig} 
\graphicspath{ {./images/} }
\usepackage{rotating}
\usepackage{bm}
\usepackage{verbatim}
\usepackage{xcolor}
\usepackage[dvipsnames]{xcolor}
\usepackage[normalem]{ulem}
\usepackage{pgfplots} 
\usepgfplotslibrary{fillbetween}
\usetikzlibrary{patterns}

\newcommand{\Z}{\mathbb{Z}}

\newcommand{\R}{\mathbb{R}}

\newcommand{\RP}{{\mathbb{RP}}{}}

\newcommand{\lan}{\langle}
\newcommand{\ran}{\rangle}

\newcommand{\disp}{\displaystyle}

\newcommand{\Mod}{\operatorname{Mod}}

\newcommand{\Diff}{\operatorname{Diff}}

\usepackage{soul}\usepackage[T1]{fontenc}
\usepackage{mathtools}
\DeclareMathAlphabet{\mathbbmsl}{U}{bbm}{m}{sl}

\newtheorem{theorem}{Theorem}[section]
\newtheorem{corollary}[theorem]{Corollary}
\newtheorem{lemma}[theorem]{Lemma}
\newtheorem{proposition}[theorem]{Proposition}

\theoremstyle{definition}
\newtheorem{definition}[theorem]{Definition}
\theoremstyle{definition}
\newtheorem{example}[theorem]{Example}
\theoremstyle{definition}
\newtheorem{remark}[theorem]{Remark}
\theoremstyle{definition}

\usepackage{tikz}
\usepackage{tikz-cd}
\usetikzlibrary{positioning}
\usetikzlibrary{knots}
\usetikzlibrary{decorations.pathreplacing}
\usetikzlibrary{intersections}

\usepackage{hyperref}
\usepackage{hhline}

\newcommand{\CPsum}[2]{#1 \mathbb{CP}^2 \# #2 \overline{\mathbb{CP}^2}}
\newcommand{\Fibersumd}[3]{#1 \#_{#3} #2} 
\newcommand{\set}[1]{\{#1\}} 

\renewcommand{\phi}{\varphi}

\author[Mihail Arabadji]{Mihail Arabadji}
\address{Department of Mathematics and Statistics, University of Massachusetts, Amherst, MA 01003-9305, USA}
\email{marabadji@umass.edu}

\author[Porter Morgan]{Porter Morgan}
\address{Department of Mathematics and Statistics, University of Massachusetts, Amherst, MA 01003-9305, USA}
\email{pamorgan@umass.edu}

\title[Even 4-manifolds with order two $\pi_1$]{Irreducible 4-manifolds with order two fundamental group and even intersection form}

\begin{document}	
	\begin{abstract}
		We construct smooth manifolds with order two $\pi_1$ and even intersection forms which are \emph{irreducible}, meaning they do not decompose into non-trivial connected sums. Their intersection forms being even implies that their universal covers admit spin structures. Such manifolds are determined up to homeomorphism by their Euler characteristic $e$, signature $\sigma$, and whether they themselves are also spin. In the case that the manifold is spin, we construct irreducible manifolds for all but $17$ realizable coordinates in the region of the $(e,\sigma)$--plane with $c_1^2 = 2e+3\sigma \geq 0$ up to orientation. In the case that the manifold is non-spin, we construct irreducible manifolds for all but $24$ realizable coordinates in the region of the $(e,\sigma)$--plane with $\sigma/8<-8$ and $c_1^2/4>9$, again up to orientation. We construct these manifolds by taking equivariant fiber sums of Lefschetz fibrations and other symplectic manifolds which are simply--connected and spin. Along the way, we develop machinery to track when the spin structure is preserved during these operations.
	\end{abstract}
\maketitle

\section{Introduction}

A central question in the topology of four-dimensional manifolds is their classification up to homeomorphism. Research in this direction is largely propelled by the celebrated results of Freedman, who showed that the homeomorphism type of a smooth, closed, simply--connected 4-manifold is determined by its intersection form. Thus, determining the homeomorphism type of a smooth, simply--connected $4$--manifold is reduced to computing its Betti numbers and the parity of its intersection form. Considerable attention in $4$-dimensional topology has been devoted to realizing any $(a,b) \in \Z_{\geq 0} \oplus \Z$ by a smooth, closed 4-manifold whose second Betti number is $a$ and signature is $b$. This is referred to as a \textit{geography problem}, and many variations of it have been formulated in recent decades. In general, the aim is to ``realize'' each coordinate in your \emph{geography plane}, your set of reachable points, by constructing a $4$--manifold whose topological invariants match the coordinate values. The ultimate goal is to construct interesting examples of manifolds for each homeomorphism class in your plane. For this reason, you might impose that in order for a manifold $M$ to ``count'' as realizing a coordinate in your plane, $M$ must be be \textit{irreducible}, meaning that if $M$ is diffeomorphic to a connected sum $A\# B$, either $A$ or $B$ is homeomorphic to $S^4$. Irreducible manifolds often  carry exotic smooth structures; $M$ is an \textit{exotic} $N$ whenever $M$ and $N$ are homeomorphic but not diffeomorphic.

There has been a lot of progress in the simply-connected irreducible geography. Various constructions of simply--connected manifolds were shown to be irreducible as early as the 90's. Later, constructions were given that covered entire regions of the geography plane, such as the symplectic manifolds in \cite{Park_Szabo} and \cite{ABBKP}. Departing from the simply-connected case, Torres provided irreducible constructions for the geographies of closed, symplectic, manifolds with abelian fundamental groups of two generators  in \cite{Torres_nonspin} and \cite{torres_spin}. When working in the symplectic category, the geography plane is restricted to coordinates with $b_2^+$ odd. More recently, Baykur, Stipsicz, and Szabó \cite{baykur2024smooth} gave an extensive treatment to the geography problem for manifolds with finite cyclic fundamental group and odd intersection forms. Unlike previous constructions, their constructions are not always symplectic, so their geography plane includes all $b_2^+$ values, regardless of parity.

We aim to address the irreducible geography problem for manifolds with order two fundamental group. Our constructions are not necessarily symplectic, so may have $b_2^+$ odd or even. That said, we prove that our constructions are irreducible by showing that their double covers are minimal symplectic. Since minimal symplectic manifolds have non-negative first Chern number \cite{Taubes}, our geography plane also inherits this constraint. 

Unlike the simply--connected case, the homeomorphism type of closed, smooth $4$--manifolds with order two fundamental group is \emph{not} determined entirely by the intersection form. The manifold's \emph{$w_2$--type}, which depends on whether it and its universal cover admit spin structures, must also be taken into consideration. There are three possible $w_2$-types, which in turn give three irreducible geography problems. The irreducible geography for $w_2$-type $(i)$ was addressed in \cite{ArabadjiMorgan}, where we found irreducible manifolds for each coordinate in our geography plane except for seven coordinates. In the current paper, we cover the remaining two geographies, types $(ii)$ and $(iii)$. These are $4$-manifolds with spin double covers. Equivalently, these are $4$--manifolds with even intersection forms. Hence, all manifolds constructed in this paper have an intersection form isomorphic to block sums of the standard even matrices $-E_8$ and $H$ (see \cite[Section 1.2]{GompfStip} for more exposition on these matrices). We start with our results for $w_2$-type $·(ii)$, which are manifolds admitting spin structures.

\begin{theorem}\label{thm:type_2_main}
	Let $a$ and $b$ be non-negative integers satisfying $b\geq 4a-1$. Then for all except $17$ $(a,b)$ coordinates, there exists an irreducible, closed, spin $4$--manifold with order two fundamental group whose intersection form is $2a (\pm E_8)\oplus bH$.
\end{theorem}

The condition $b\geq 4a-1$ is equivalent to $X_{a,b}$ having a non-negative first Chern number. This is a consequence of the double cover being minimal symplectic, which is how we prove irreducibility. Next we state our results for $w_2$ type $(iii)$. These are manifolds which are \emph{not} spin, but whose double covers are spin.

\begin{theorem}\label{thm:type_3_main}
	Let $a$ and $b$ be non-negative integers such that \begin{enumerate}[(i)]
		\item if $b$ is odd, assume $a\neq 0$, and $b\geq 2a-1.$
		\item if $b$ is even, assume $a\notin \{0,1,2,4,6,8\}$ and $b>2a+6$.
	\end{enumerate} Then for all but $24$ $(a,b)$ pairs, there exists an irreducible, closed, orientable, non-spin $4$--manifold with order two fundamental group whose intersection form is $a(\pm E_8)\oplus bH$.
\end{theorem}

The value $b$ is the same as the manifold's $b_2^+$ value. When $b_2^+$ is odd, we give irreducible constructions for the entire $c_1^2\geq 0$ geography plane except for the line $\sigma = 0$ and finitely many missing points. Constructing irreducible manifolds with $w_2$-type $(iii)$ and $b_2^+$ even has proven to be more challenging. In this case, coordinates with $\sigma \equiv 0 \bmod 16$ and $c_1^2 \in \{4, 12, 20, 28, 36\}$, as well as coordinates with $\sigma \in \{0,8,16,32, 48, 64\}$, are yet to be realized. We leave these regions for future work. Just like for $w_2$-type $(ii)$, we will rely on our constructions having non-negative first Chern numbers to prove irreducibility.

\subsubsection*{Our contribution} Prior to this work, there was already progress constructing irreducible \mbox{$4$-manifolds} with order two fundamental group and even intersection forms. For the manifolds with $w_2$-type $(ii)$ in Theorem \ref{thm:type_2_main}, many of the intersection forms with odd $b_2^+$ (meaning odd $b$ in the theorem statement) have already been realized by Torres in \cite{torres_spin}, who built minimal symplectic $4$-manifolds with abelian fundamental groups using torus surgery. Other methods for building irreducible smooth structures for this class of manifolds were given in \cite{stipsiczszabdo2023}, also having $b_2^+$ odd. Because of this, our main contribution in Theorem \ref{thm:type_2_main} is our constructions for intersection forms of the form $2a (\pm E_8)\oplus bH$ with $b$ even. The manifolds of $w_2$-type $(iii)$ in Theorem \ref{thm:type_3_main} have received less attention. Examples of previous constructions include \cite{OKONEK, wang,stipsiczszabdo2023}, which have odd $b_2^+$ and low $c_1^2$. To our knowledge, most of our constructions for Theorem \ref{thm:type_3_main} are the first irreducible smooth structures in their homeomorphism classes.

\subsubsection*{Methods} Our main construction tool in this paper is the $\mathbbmsl{Z}_2$\emph{--construction}. This procedure, pioneered in \cite{baykur2024smooth}, involves taking equivariant fiber sums to construct minimal symplectic simply--connected manifolds equipped with free orientation-preserving involutions. We employ this technique on a variety of $4$--manifolds, including Lefschetz fibrations as well as previous constructions of spin, simply--connected, symplectic manifolds such as in \cite{Park_Szabo} and \cite{baykurhamadaspin}. We often fiber sum a few of these constructions together before performing the $\Z_2$--construction, and in this sense they serve as our ``building blocks''. When working with Lefschetz fibrations, we sometimes use torus surgery to reduce the fundamental group of the total space. Because all of our manifolds have spin double covers, we need to ensure that the spin structures of our building blocks are preserved throughout these operations. We lay out groundwork for this in Section \ref{sec:spin}.

A necessary tool for us is one that determines the $w_2$-type of a manifold, which will be formally defined in Section \ref{sec:background}. To show that a manifold has $w_2$-type $(i)$, it suffices to find an embedded surface of odd self-intersection. By contrast, discerning between $w_2$-types $(ii)$ and $(iii)$ requires more information than the manifold's intersection form. We address this issue by extending Wu's formula to the $2$-torsion case, which gives us conditions in terms of $\Z_2$-homology that distinguish $w_2$-types $(ii)$ and $(iii)$. Using this formula, we're able to discern when the $\Z_2$--construction of a Lefschetz fibration will have have $w_2$-type $(ii)$ or $(iii)$. This tool will be introduced in Section \ref{sec:wu_extension} and used throughout our constructions.

\subsubsection*{Outline of paper} We begin by providing the necessary definitions and background information in Section \ref{sec:background}. Then in Section \ref{sec:spin}, we provide a detailed analysis of spin structures on $4$--manifolds, and how they behave under the $\Z_2$--construction and torus surgery. We then begin our constructions.  We have two different geographies to fill depending on the $w_2$-type, and we break these two geographies up further into two cases: $b_2^+$ even and $b_2^+$ odd. Section \ref{sec:type2} covers the $w_2$-type $(ii)$ geography, which is for spin manifolds. Section \ref{sec:type2_even} handles the constructions for $b_2^+$ even, and Section \ref{sec:type2_odd} handles $b_2^+$ odd. Then in Section \ref{sec:type3}, we move on to $w_2$-type $(iii)$, which are non-spin manifolds with spin double covers. Similar to the previous geography, Section \ref{sec:type3_even} covers the $b_2^+$ even case and Section \ref{sec:type3_odd} covers the $b_2^+$ odd case.  For both $w_2$-types, it's generally easier to find constructions in the $b_2^+$ odd case than the $b_2^+$ even case. We end each of these four blocks of constructions with graphs showing which parts of our geography were reached.

\subsubsection*{Acknowledgments} The authors would like to thank their advisor, \.{I}nan\c{c} Baykur, for his guidance and support. They are also grateful to Zoltán Szabó for his feedback on Section \ref{sec:spin}. Thanks also to Kerem \.{I}nal for many helpful conversations, and to Rafael Torres for insightful emails.

\section{Background}\label{sec:background}

\subsubsection*{Conventions} All $4$--manifolds will be assumed to be orientable, unless otherwise stated. The mapping class group of a surface $F$ is denoted $\Mod(F)$, and $t_c$ denotes a right-handed Dehn twist about a simple closed curve $c$.

\subsection{Topological classification of manifolds with small $\mathbf{\pi_1}$.}

Freedman showed that a simply-connected, smooth, closed 4-manifold $M$ is classified up to homeomorphism by its intersection form $Q_M$. The intersection form is itself determined by its number of positive and negative eigenvalues, along with its parity, which is even if $Q_M(a,a)$ is even for all $a\in H_2(M,\Z)$ and odd otherwise. Freedman's theorem has been generalized to other fundamental groups by Hambleton-Kreck. Their classification for manifolds with $\pi_1=\Z_2$ is as follows.

\begin{theorem}[\cite{HambletonKreck}]
	Let $M$ be a closed, smooth, oriented 4-manifold with order two fundamental
	group. Then M is classified up to homeomorphism by its
	intersection form on $H_2(M,\Z)/Tor$ and its $w_2$-type.
\end{theorem}

Here, the $w_2$-type of $M$ can have 3 values, depending on the second Stiefel-Whitney classes of $M$ and its double cover $\widetilde{M}$; type $(i)$ means that $w_2(M)$ and $w_2(\widetilde{M})$ are both non-zero, type $(ii)$ means that $w_2(M)\equiv w_2(\widetilde M)\equiv 0$, and type $(iii)$ means that $w_2(M)\neq 0$ and $w_2(\widetilde M)=0$. Note that for an orientable manifold, having vanishing second Stiefel-Whitney class is equivalent to admitting a spin structure. 

This paper is concerned with manifolds whose intersection forms are even (and hence are, up to isomorphism, block sums of copies of $H$ and $\pm E_8$). Given this, the following formulas for the algebraic invariants of even $4$-manifolds will be handy. In a classical sense, $\chi_h$ denotes a holomorphic Euler characteristic, which requires an almost complex structure. We work in the topological setting, where the convention is  $\chi_h(M):=\frac{e(M)+\sigma(M)}{4}$. This agrees with the classical definition when an almost complex structure is specified \cite[Section 1.4]{GompfStip}.

\begin{proposition}
	Given a closed $4$--manifold $X$ with $b_1(X)=0$ and $Q_X = n(-E_8)\oplus \ell H $ with $n,\ell \in \Z_{\geq 0}$, we have the following algebraic invariants for $X$:
	\begin{align*}
		& e(X) = 2+2\ell +8n & \sigma(X) = -8n\\
		& b_2^+(X) = \ell & b_2^-(X) = \ell +8n\\
		& c_1^2(X) = 4+4\ell -8n & \chi_h(X) = (1+\ell)/2
	\end{align*}
\end{proposition}

\subsection{Symplectic machinery} \label{sec:symp_background} There are a few topics from symplectic topology that we need to review. We begin with Lefschetz fibrations, which are important examples of symplectic manifolds. For more exposition on this family of fibrations, see \cite[Chapter 8]{GompfStip}.

\begin{definition}
	Let $X$ and $\Sigma$ be an oriented 4-manifold and surface, respectively. A smooth surjection $f\colon X\to \Sigma$ is called a Lefschetz fibration if it has finitely many critical points, and around each critical point there are orientation--preserving charts where $f$ is locally modeled by the map $g\colon\mathbb{C}^2\to \mathbb{C}$ defined as $g(z_1,z_2)=z_1z_2$. The genus of $f$ refers to the genus of a regular fiber.
\end{definition}

To a Lefschetz fibration $f\colon X\to \Sigma$ with regular fiber $F$ and $n$ singularities, there's a prescribed composition of Dehn twists, $\mu := t_{c_n}\circ t_{c_{n-1}}\circ \cdots \circ t_{c_1}$ in $\Mod(F)$. This composition is the \emph{global monodromy} associated to $f$, and each singular fiber is obtained from $F$ by collapsing the simple closed curve $c_i\subset F$ to a point. The curves $\{c_i\}$ are the \emph{vanishing cycles} of the Lefschetz fibration. If $\partial \Sigma = S^1$, then the Lefschetz fibration can be capped off with a copy of $F\times D^2$ to obtain a Lefschetz fibration on a closed $4$--manifold precisely when $\mu = 1$ in $\Mod(F)$. Additionally, a \emph{section} of $f$ is a map $s:\Sigma\hookrightarrow X$ such that $f\circ s:\Sigma\to \Sigma$ is the identity.

Stipsicz \cite{StipsiczSpinLF} and later Baykur, Hamada \cite{baykurhamadaspin} showed how to determine if a Lefschetz fibration over any genus surface admits a spin structure by examining its monodromy and sections. Recall that a \emph{quadratic form} on an orientable surface $F$ is a map $q\colon H_1(F;\Z_2)\to \Z_2$ satisfying $q(c_1+c_2)=q(c_1)+q(c_2)+c_1\cdot c_2$, where $c_1\cdot c_2$ denotes the algebraic intersection. Quadratic forms are used to define spin structures on surfaces, and may also be used to define spin structures on Lefschetz fibrations. We will use the following theorem over $D^2$. 

\begin{theorem}\label{thm:stip1}
	The total space of a Lefschetz fibration $f\colon X\to D^2$ with regular fiber $F$ admits a spin structure if and only if there exists a quadratic form on $F$ evaluating to one on each vanishing cycle of $f$.
\end{theorem}

This is generalized in \cite{StipsiczSpinLF} and \cite{baykurhamadaspin} to Lefschetz fibrations over $S^2$. 

\begin{theorem}\label{thm:stip2}
	
	Let $f\colon X\to S^2$ be a Lefschetz fibration with regular fiber $F$, where $F$ has a geometric dual of even self intersection. Then if $X$ admits a spin structure if and only if $X-\nu(F)$ does.
\end{theorem}

Given the Lefschetz fibration $f\colon X\to S^2$, you obtain a Lefschetz fibration over $D^2$ when you remove $\nu(F)$, and can use Theorem \ref{thm:stip1} to check if $X-\nu(F)$ is spin. Then if $f$ has a section of even self-intersection, this section would be dual to $F$, and so confirms that $X$ is also spin. This criteria for determining when Lefschetz fibrations are spin will be used in Section \ref{sec:type3}.

The elliptic fibrations $E(n)$ are examples of Lefschetz fibrations that are ubiquitous in our constructions, so we remind the reader of their properties. For each $n\geq 1$, $E(n)$ is the total space of a genus one Lefschetz fibration with $12n$ singularities. It is simply--connected with Euler characteristic $12n$ and signature $-8n$. Its intersection form is isomorphic to $n(-E_8)\oplus (2n-1) H$ when $n$ is even and $(2n-1)\lan 1\ran \oplus (10n-1) \lan -1\ran $ when $n$ is odd. This demonstrates that $E(n)$ is an even manifold (and hence spin) when $n$ is even, and is odd when $n$ is odd. Interestingly, the quadratic form evaluating to one on both symplectic generators of $H_1(T^2;\Z_2)$ defines a spin structure $E(n)-\nu(F)$ for all $n$, where $F$ is a regular torus fiber. So although the total space $E(n)$ is only spin half the time, $E(n)-\nu(F)$ is always spin. While $E(1)$ decomposes as $\CPsum{}{9}$, $E(n)$ is irreducible for all other $n>1$.

Given a Lefschetz fibration $f\colon X\to \Sigma$ with homologically essential fibers, there's a natural way to equip $X$ with a symplectic form such that fibers and sections are symplectic. All of our building blocks in later sections will be symplectic manifolds, and many of them will come equipped with a specific Lefschetz fibration. We say that a symplectic manifold is \textit{minimal} if it does not contain an embedded sphere with self-intersection $(-1)$. The following theorem conveys the utility of symplectic minimal manifolds for our purposes.

\begin{theorem}[Hamilton and Kotschick] \cite{Hamilton_2006}\label{thm:HK_minimal}
	Minimal symplectic\/ $4$-manifolds with residually finite fundamental groups are irreducible.
\end{theorem}

From this we immediately deduce the following.

\begin{lemma}\label{lem:symplectic_irreducible}
	If $X$ is symplectic and spin with a finite fundamental group, then $X$ is irreducible. Moreover, if $X$ has finite $\pi_1$ and has a symplectic minimal universal cover, then $X$ is irreducible.
\end{lemma}

During our constructions, there are two operations we will perform on symplectic manifolds to build new ones. The first of these is a \emph{symplectic fiber sum,} introduced by Gompf in \cite{GompfNew}. Let $(M_1,\omega_1)$ and $(M_2,\omega_2)$ be symplectic $4$--manifolds where each $M_i$ admits an embedded genus $g$ symplectic surface $\Sigma_i$ such that $[\Sigma_i]^2 =0$. We can obtain a new symplectic manifold by gluing $M_1 \setminus \nu(\Sigma_1)$ and $M_2\setminus \nu(\Sigma_2)$ along the boundaries $\partial \nu(\Sigma_i)$ via an orientation--reversing bundle diffeomorphism. This new manifold is a {\it symplectic fiber sum} of $M_1$ and $M_2$, and its symplectic form restricts to $\omega_i$ on each $M_i\setminus \nu(\Sigma_i)$. We denote it by $M_1\#_{\Sigma_1=\Sigma_2} M_2$, or sometimes $M_1\#_\Sigma M_2$ when the gluing surfaces are clear from the context. This notation suppresses the choice of gluing map, which will be specified separately when needed. The symplectic fiber sum of two manifolds along a genus $g$ surface has algebraic invariants given by $e(M_1\#_{\Sigma_g} M_2) = e(M_1)+e(M_2) +2-2g$ and $\sigma(M_1\#_{\Sigma_g} M_2) = \sigma(M_1)+\sigma(M_2)$. We will use this operation to construct a new symplectic manifold with different Euler characteristic and signature than the manifolds it's built from.

While symplectic fiber sums are effective for obtaining irreducible $4$--manifolds with various Euler characteristics and signatures, we introduce another operation which preserves these invariants but changes $\pi_1$. Let $X$ be a smooth $4$--manifold and $T$ an embedded torus with framing $\eta\colon \nu T\to T^2\times D^2$. Let $\lambda$ be the image of a push-off of a primitive curve on $T$ under $\eta$, and let $\mu_T$ be a meridian of $T$, i.e. a fiber of $\partial \nu T$. Then let $\phi\colon T^2\times S^1 \to \partial \nu T$ be a diffeomorphism which induces $[\{pt\}\times S^1]\mapsto p[\mu_T]+q[\lambda]$ on the first homology. The manifold $X'$ given by \[
	X' = (X\setminus \nu T) \cup_\phi T^2\times D^2,
\] 
is the result of {\it $p/q$-torus surgery} performed on $X$ along $T$ with respect to $\lambda$. From the Seifert-van Kampen theorem, we see that $\pi_1(X')$ is obtained from $\pi_1(X-T)$ by quotienting out the subgroup normally generated by $\mu_T^p\lambda_\eta^q$. The 4-manifold obtained from this construction admits symplectic structure whenever $X$ is symplectic, $T$ is Lagrangian, and $p=1$. In this case, the surgery is called {\it Luttinger surgery with coefficient $1/q$.} Luttinger surgery was first introduced in \cite{Luttinger} and then generalized in \cite{AurouxDonaldsonKatzarkov2003}.

\subsection{$\mathbf{\Z_2}$--construction}

The following construction is introduced in \cite{baykur2024smooth}, which we review here. Let $Z$ be a closed, oriented, smooth 4-manifold and let $\Sigma$ be an embedded closed genus $g$ surface in $Z$ with a trivial normal bundle. Let $\eta \colon \nu\Sigma\to D^2\times \Sigma_g$ be a framing of the normal bundle for $\Sigma$. Then let $\phi$ be an orientation-reversing, free, involutive bundle map of $S^1\times \Sigma_g$. The map \[\Phi:= \eta|_{\partial(\nu\Sigma)}^{-1}\circ \phi \circ \eta|_{\partial(\nu\Sigma)}\] is an orientation-reversing free involution on $\nu\Sigma$. The manifold \[\widetilde{X}:=(Z\setminus \nu\Sigma)\cup_{\Phi} (Z\setminus \nu\Sigma)\] is the {\it double} of $Z$ along $\Sigma$. Let $\tau$ be the orientation-preserving free involution on $\widetilde{X}$ that swaps the two copies of $Z\setminus \nu\Sigma$ and restricts to $\Phi$ on $\partial(\nu\Sigma)$. Then $X:=\widetilde{X}/\tau$ is a smooth, oriented, closed 4--manifold which we call the $\mathbbmsl{Z}_2$\emph{--construction of $Z$ along $\Sigma$.} We will employ some variation of this technique in all of our constructions. In all cases, the double $\widetilde{X}$ will admit a spin structure. When constructing manifolds with $w_2$-type $(ii)$, the $\Z_2$--construction $X$ will also be spin, and this will not be the case for $w_2$-type $(iii)$. The following lemma shows how algebraic invariants change when the $\Z_2$--construction is performed. The proof is a matter of simple algebraic computation and is omitted.

\begin{lemma}\label{lemma:inv_computations}
	
	Let $M$ be a $4$--manifold with $b_1(M)=0$ and let $\Sigma$ an embedded genus $g$ surface. Then taking a double/quotient/both along $\Sigma\subset M$ transforms the invariants of $M$ as follows, respectively:
	\begin{align*}
		\textit{Double} && \textit{Quotient} && \textit{$\Z_2$--construction}\\  
		\sigma \mapsto 2\sigma && \sigma \mapsto \sigma/2 && \sigma \mapsto \sigma \\
		e \mapsto 2e + 4g- 4 && e \mapsto e/2 && e \mapsto e+2g-2\\ 
	\end{align*}
\end{lemma}

\subsection{What are we looking for, anyways?}\label{sec:history}

This paper aims to find irreducible 4-manifolds with order two $\pi_1$ and $w_2$--types $(ii)$ and $(iii)$. Before beginning our constructions, we digress to motivate the constraining inequalities on the geography plane for these  $w_2$--types.

Let $X$ be a closed, smooth, oriented $4$--manifold with order two $\pi_1$ and $w_2$ type $(ii)$ or $(iii)$. In either case, the intersection form of $X$ is even. It suffices to restrict our attention to the case where $\sigma(X)\leq 0$, since one can always reverse orientation to cover the $\sigma>0$ geography. The intersection form of $X$ is given by $a(-E_8)\oplus b H$ for some $a,b\geq 0$  \cite[Chapter 1]{GompfStip}. Let $\widetilde X$ be the spin, simply--connected double cover of $X$. It's not hard to compute that the intersection form of $\widetilde X$ is then $2a(-E_8)\oplus (2b+1)H$. Set $x=2a$ and $y=2b+1$, so that $Q_{\widetilde X} = x(-E_8)\oplus y H$. From Furuta's theorem, we know that $b_2(\widetilde X)\geq \frac{5}{4}\left|\sigma(\widetilde X)\right|+2$. This gives the inequality $y\geq x+1$, which is equivalent to $b\geq a$. The famous $11/8$ conjecture gives a stricter bound: $b_2(\widetilde X)\geq \frac{11}{8}\left|\sigma(\widetilde X)\right|$. This is equivalent to $y\geq \frac{3}{2} x$ or $b\geq \frac{3a-1}{2}$. In order to argue irreducibility, we will construct manifolds with minimal symplectic universal covers. Taubes' theorem shows that all minimal, symplectic, non-ruled $4$--manifolds have a non-negative first Chern number, $c_1^2$ \cite{Taubes}. Equivalently, $2e(\widetilde X)+3\sigma(\widetilde X)\geq 0$. This adds the constraint $y\geq 2x-1$ to our geography, which is the same as $b\geq 2a-1$.

These inequalities constrain the intersection forms that we hope to find here. Namely, we will content ourselves with constructing manifolds having intersection forms isomorphic to $a(-E_8)\oplus bH$, where $a,b\geq 0$ and $b\geq 2a-1$. Figure \ref{fig:118_and_furuta} shows how Furuta's theorem, the $11/8$--conjecture, and the $c_1^2\geq 0$ assumption restrict the $(x,y)$ coordinates of $Q_{\widetilde X} \cong x (-E_8)\oplus yH$ and the $(a,b)$ coordinates of $Q_X\cong a(-E_8) \oplus bH$, respectively. Note that if $X$ is also spin, we may apply the constraints coming from Furuta's theorem, the $11/8$ conjecture, and Taubes theorem directly to $X$ rather than $\tilde X$. So in this case, the conditions on $Q_X$ would be the same as those for $Q_{\tilde X}$, shown on the left in Figure \ref{fig:118_and_furuta}. However, regardless of whether $X$ has $w_2$-type $(ii)$ or $(iii)$, the strictest constraint, the one coming from Taubes' theorem, is the same: $b\geq 2a-1$. Hence in both Sections \ref{sec:type2} and \ref{sec:type3}, we populate the region spanned by $(a,b)$ pairs satisfying this inequality.

\begin{figure}[h]
	\centering
	\resizebox{6.5in}{!}{
		\begin{tikzpicture}

			\begin{scope}[xshift = -5cm]
				
				\node[black, scale = 1.2] at (3,6.2) 
				{Intersection form of $\widetilde X$};
				
				\begin{axis}[
					axis lines = middle,
					xmin=0, xmax=15,
					ymin=0, ymax=15,
					xtick = {0,1,...,15},
					ytick = {0,1,...,15}, 
					grid=both,
					ticklabel style = {font=\tiny},
					every axis plot/.append style={thick}]
					
					\draw [red!90!black, very thick] (axis cs:0,1) -- (axis cs:14,15); 
					
					\draw [blue!90!black, very thick] (axis cs:0,0) -- (axis cs:10,15); 
					
					\draw[green!60!black, very thick] (axis cs:0,-1) -- (axis cs: 8,15);
					
					\draw(axis cs:14.5,0) node[anchor = south,scale=.9] {$x$};
					\draw(axis cs:0,14.5) node[anchor=west,scale=.9] {$y$};
					
				\end{axis}
				
			\end{scope}

			\begin{scope}[xshift = 5cm]
				
				\node[black, scale = 1.2] at (3,6.2) 
				{Intersection form of $X$};
				
				\begin{axis}[
					axis lines = middle,
					xmin=0, xmax=15,
					ymin=0, ymax=15,
					xtick = {0,1,...,15},
					ytick = {0,1,...,15}, 
					grid=both,
					ticklabel style = {font=\tiny},
					every axis plot/.append style={thick}]
					
					\draw [red!90!black, very thick] (axis cs:0,0) -- (axis cs:15,15); 
					
					\draw [blue!90!black, very thick] (axis cs:0,-1/2) -- (axis cs:31/3,15); 
					
					\draw[green!60!black, very thick] (axis cs:0,-1) -- (axis cs: 8,15);				
					
					\draw(axis cs:14.5,0) node[anchor = south,scale=.9] {$a$};
					\draw(axis cs:0,14.5) node[anchor=west,scale=.9] {$b$};

				\end{axis}
				
			\end{scope}
			
			\draw(3.25,1) node {\huge $\longrightarrow$};
			
			\draw(3.25,4.5) node[red!90!black] {Furuta};
			\draw[-{Stealth}, red!90!black] (2.5,4.5)--(.2,4.5);
			\draw[-{Stealth}, red!90!black] (4,4.5)--(10.2,4.5);
			
			\draw(3.25,3.5) node[blue] {$11/8$};
			\draw[-{Stealth}, blue] (2.75,3.5)--(-2,3.5);
			\draw[-{Stealth}, blue] (3.75,3.5)--(7.75,3.5);
			
			\draw(3.25,2.5) node[green!60!black] {$c_1^2\geq 0$};
			\draw[-{Stealth}, green!60!black] (2.5,2.5)--(-3.2,2.5);
			\draw[-{Stealth}, green!60!black] (4,2.5)--(6.5,2.5);
		\end{tikzpicture}
	}
	\caption{The lines on the left restricting $Q_{\tilde X}$ are $y\geq x+1$, $y\geq \frac{3}{2}x$, and $y\geq 2x-1$. The lines on the right restricting $Q_{X}$ are $b\geq a$, $b\geq \frac{3a-1}{2}$, and $b\geq 2a-1$. Recall that $x=2a$ and $y=2b+1$.}
	\label{fig:118_and_furuta}
\end{figure}
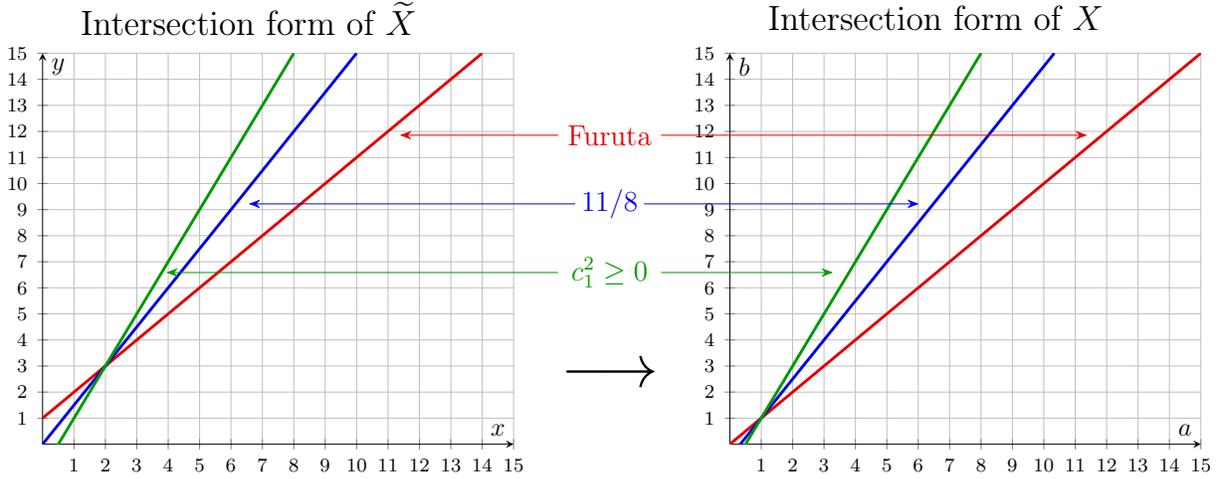

\section{Spin structures on the $\Z_2$--construction}\label{sec:spin}

The goal of this section is to prove the following.

\begin{theorem} \label{thm:Z2_spin}
	Let $X$ be a smooth, spin $4$--manifold and $\Sigma$ an embedded orientable surface with a trivial normal bundle, such that $X-\nu \Sigma$ is simply--connected. Then there is an orientation-reversing involutive diffeomorphism $\phi:S^1\times \Sigma \to S^1\times \Sigma$ such that the $\Z_2$--construction of $X$ along $\Sigma$ with gluing map $\phi$ is spin (i.e., has $w_2$--type $(ii)$).
\end{theorem}

Proving Theorem \ref{thm:Z2_spin} will involve a close examination of how spin structures are defined, constructed, and preserved. The reader is referred to \cite[Chapter 4]{scorpan2005} for more expository information on spin structures.

\subsection{Non-orientable Wu formula}\label{sec:wu_extension}

Recall the well-known formula from Wu, which provides useful criteria relating an even intersection form to the existence of a spin structure.

\begin{theorem}[Wu's Formula]
	For all oriented surfaces $S$ embedded in a smooth $4$--manifold $M$, $w_2(M)\cdot S = S\cdot S \bmod 2$, where $\cdot$ denotes the algebraic intersection.
\end{theorem}
	
This implies that if $H_2(M)$ is generated by orientable surfaces, which happens whenever $H_1(M;\Z)$ has no 2--torsion, then if $Q_M$ is even, $M$ is spin. On the other hand, if $H_1(M)$ has 2-torsion, then $H_2(M)$ may not be generated by orientable surfaces, so we can't appeal to the parity of $Q_M$ to claim that it's spin. The following proposition extends Wu's theorem to all classes in $H_2$ taken with $\Z_2$ coefficients.

\begin{proposition}
	If $M$ is an orientable, closed, smooth $4$-manifold, then $w_2(X)(S) = [S]^2$ for all $S\in H_2(M;\Z_2)$.
	
	\label{prop:w_2_nonorientable_surface}
\end{proposition}

The proof of this proposition uses the Wu class $v\in H^*(M;\Z_2)$, which is characterized by the identity $\langle v\smile x, [M] \rangle=\langle Sq(x),[M]\rangle $, where $[M]$ is the fundamental class and $Sq$ is the Steenrod square operator. The proof also relies on Wu's theorem, which relates the Wu class and the Stiefel-Whitney class. For more exposition on these topics, see \cite[Chapter 11]{milnor_stasheff}. 

\begin{proof}[Proof of Proposition \ref{prop:w_2_nonorientable_surface}]
	Let $w=\sum_{i=0}^4 w_i$ and $v=\sum_{i=0}^4 v_i$ denote the Stiefel Whitney and Wu classes of $M$, respectively. From Wu's formula $w_k = \sum_{i+j=k} Sq^i(v_j)$ and the fact that $w_1=0$, we deduce that $w_2=v_2$. Now let $S\in H_2(M;\Z_2)$ and let $S^\vee$ be its dual in $H^2(M;\Z_2)$. The Wu class satisfies\begin{align*}
		\langle v_2\smile S^\vee,[M]\rangle &= \langle Sq^2(S^\vee), [M]\rangle\\
		&= \langle S^\vee \smile S^\vee,[M]\rangle \\
		&= [S]^2.
	\end{align*} 
	Hence $\langle w_2\smile S^\vee,[M]\rangle = [S]^2$, completing the proof.
\end{proof}

\begin{corollary}\label{cor:Z2_spin}
	Let $M$ be a closed, smooth $4$--manifold. If there exists a basis $\{b_1,b_2,\cdots, b_n\}$ for $H_2(M;\Z_2)$ such that $b_i^2= 0\bmod 2$ for all $i$, then $M$ is spin.
\end{corollary}

As an application of Corollary \ref{cor:Z2_spin}, we give conditions for torus surgery to preserve spin structures. The proposition below extends the idea from the proof \cite[Theorem 10]{baykurhamadaspin} to the case with torsion in $H_1(M;\Z)$. In the following proposition all Betti numbers are in $\Z_2$.

\begin{proposition} \label{prop:dual_torus_spin}
	Let $M$ be a spin 4-manifold, and $M_T$ be the outcome of a $1/1$-torus surgery of $M$ along an embedded torus $T$ such that $b_1(M_T) = b_1(M)-1$ (in $\Z_2$ coefficients). Assume that $T$ has a dual $T_d$ in $M$ such that $\langle T, T_d \rangle$ generate a $2\times 2$ minor $H$ in the intersection form over $H_2(M,\Z)$. Then $M_T$ is spin.
\end{proposition}

\begin{proof}

	From $e(M_T)=e(M)$, we get $b_2(M_T) = b_2(M)-2$. Now we will show that $H_2(M_T;\Z_2)$ is even. In what follows, $Q_X$ denotes the intersection form of a manifold $X$ with $\Z_2$ coefficients. Pick a basis $\{g_1,\cdots g_n, T, T_d \}$ for $H_2(M;\Z_2)$, so that the intersection form $Q_M$ has a $2\times 2$ hyperbolic minor $\disp \begin{bmatrix} (Q_M)_{n+1,n+1} & (Q_M)_{n+1,n+2} \\ (Q_M)_{n+2,n+1} & (Q_M)_{n+2,n+2} \end{bmatrix}= H$, and this minor is generated by $T$ and $T_d$. We will show that $Q_{M} \cong Q_{M_T} \oplus H$. We can use simultaneous row and column operations to relate the original matrix $Q_M$ to a matrix where all other entries in the $n+1$ and $n+2$ rows and columns are $0$ away from the matrix $H$. This gives a new basis $\{g'_1,\cdots g'_n, T, T_d\}$ such that $g'_i\cdot T=g'_i\cdot T_d = 0$ for all $i$. Let $N$ be the $n\times n$ minor of this new intersection matrix spanned by $\{g'_1,\cdots g'_n\}$, and note that this matrix must be non-degenerate. For each $i$, let $S_i$ be a representative of $g'_i$. Then $S_i$ intersects $T$ in an even number of points, and the same holds for $S_i$ and $T_d$. By tubing along paths on $S_i$ connecting the pairs of intersection points, we obtain a new higher--genus representative $S'_i$ of $g'_i$ which is disjoint from $T$ and $T_d$. These $S'_i$ survive in $M_T$, showing that $\{g'_1,\cdots g'_n\}$ are well-defined classes in $H_2(M_T;\Z_2)$ and their intersections are given by $N$. Hence, $N$ may be regarded as an $n\times n$ minor in $Q_{M_T}$. But since $b_2(M_T)=n$, $N$ must be all of $Q_{M_T}$. Since $Q_M=N\oplus H$ and is even, $Q_{M_T}=N$ also must be even. It follows from Corollary \ref{cor:Z2_spin} that $M_T$ is spin.
\end{proof}

\subsection{Examples: rational homology spheres}

As illustrative examples of Corollary \ref{cor:Z2_spin}, we give constructions of the manifolds $L_2$ and $L'_2$ and compute their $w_2$--types. These manifolds are already known and studied, for instance in \cite{GompfStip}, and we claim no originality with these computations. Let $r\colon S^2\to S^2$ and $a\colon S^2\to S^2$ be given by \begin{align*}
	r(u,v,w)&=(u,v,-w)\\
	a(u,v,w)&=(-u,-v,-w),
\end{align*} where we regard $S^2$ as a subspace of $\R^3$. We use these maps to define $L_2$ and $L'_2$ as follows.

\begin{example}
	Let $\phi\colon S^2\times S^2\to S^2\times S^2$ be given by $\phi(x,y)=(r(x),a(y))$. Then $L_2$ is defined as $S^2\times S^2/\phi$. We claim that the $\Z_2$-intersection form of $L_2$ is even. To see this, first note that $H_2(L_2;\Z_2) \cong \Z_2^2$.
	Consider $\tilde A := S^2 \times \set{pt} \cup S^2 \times \set{a(pt)}\subset S^2\times S^2$. Similarly, let $\tilde B := \set{pt} \times S^2\subset S^2\times S^2$ for $r(pt)=pt$. These subspaces are saturated, and their quotients under the action of $\phi$ are $A\cong S^2 $ and $B\cong \RP^2$ in $L_2$. Then $A$ and $B$ intersect at a single point. One can perturb $A$ to get $A'$ which is disjoint from $A$, same for $B$. Hence the $\Z_2$-intersection form of $L_2$ is given by $H$, and so $L_2$ has $w_2$-type $(ii)$ by Corollary \ref{cor:Z2_spin}.
\end{example}

\begin{example}
	Next let $\psi\colon S^2\times S^2\to S^2\times S^2$ be given by $\psi(x,y)=(a(x),a(y))$. Then $L'_2:= S^2\times S^2/\psi$. To see that $L'_2$ has $w_2$-type $(iii)$, we again compute its $\Z_2$-intersection form. Similar to the previous example, let $\tilde A := S^2 \times \set{pt} \cup S^2 \times \set{a(pt)}\subset S^2\times S^2$. Let $\tilde D = \set{(x,x)\mid x\in S^2}$ be the diagonal in $S^2 \times S^2$ and denote its quotient in $L'_2$ by $D$. Again, $A\cdot A=0$ and $A\cap D = \set{[(pt,pt)]}$. On the other hand, consider $\tilde D' = \set{(x,h(x)) \mid x \in S^2}$ where $h(a,b,c)= (-a,-b,c)$. Then $[\tilde D'] = [\tilde D] \in H_2(S^2 \times S^2,\Z_2)$. Let $D'$ be quotient of $\tilde D'$. Since $\tilde D'$ and $\tilde D$ intersect twice, $D\cdot D = D\cdot D'= 1$. Hence the $\Z_2$-intersection form of $L'_2$ is 
	$\begin{pmatrix}
		0& 1 \\ 
		1 & 1
	\end{pmatrix},$ so $L'_2$ has $w_2$-type $(iii)$.
\end{example}

\subsection{Doubles that preserve spin structures}

Let $M$ be a smooth, spin $4$--manifold, and $\Sigma$ an embedded orientable surface with trivial normal bundle. The first step to showing that the $\Z_2$--construction of $M$ along $\Sigma$ is spin is showing that the double of $M$ over $\Sigma$ is spin. This will at least rule out the $\Z_2$--construction having $w_2$-type $(i)$. By \cite[Proposition 1.2]{GompfNew}, one can take a fiber sum of spin 4-manifolds $M_1$ and $M_2$ along a surface in such a way that spin structures will extend to the new manifold. But this is not immediately applicable to our case because Gompf's fiber-sum gluing map preserves orientation of the $S^1$-fiber, whereas we glue $\Sigma \times S^1$ using automorphism $f\times a$, where $f$ is an orientation-reversing map on $\Sigma$ and $a$ is antipodal. We rectify this by generalizing Gompf's result to our situation.

\begin{proposition}
	Let $(M, \mathfrak{s})$ be a spin $4$-manifold and $\Sigma$ be a closed embedded orientable surface with a trivial neighborhood. Then there exists orientation reversing diffeomorphism $r:\Sigma \to \Sigma$ such that $(M-\Sigma) \cup_{r\times a} (M-\Sigma)$ has a spin structure. Moreover, even if $M$ is not spin but $M-\Sigma$ is spin, $(M-\Sigma) \cup_{r\times a} (M-\Sigma)$ is still spin.
	
	\label{prop:spin_preserving_map}
\end{proposition}

\begin{proof}
	Both copies of $M$ have the same spin structure, and so their restriction to $\Sigma \times D^2$ is the same. Since spin structures on $\Sigma\times D^2$ are classified by $H^1(\Sigma;\Z_2)$, a gluing map which induces the identity on $H^1(\Sigma;\Z_2)$ will carry the spin structure $\mathfrak{s}$ on one copy of $M-\Sigma$ to the other. If $\{x_1,y_1,\cdots x_g,y_g\}$ are the symplectic generators for $H_1(\Sigma;\Z)$, there's a simple reflection map $r$ of $\Sigma$ which fixes each $x_i$ and reverses the orientation of each $y_i$. Such a map induces the identity on $H^1(\Sigma;\Z_2)$. Thus $a\times r:S^1\times \Sigma\to S^1\times \Sigma$ gives the desired gluing, where $a$ is the antipodal map on $S^1$.
	
	To see the last claim, observe that the above proof goes through verbatim if we start with a spin structure on $M-\Sigma$, except now we're concerned with the spin structure's restriction to $\partial \nu(\Sigma)=S^1\times \Sigma$. Since the map $a\times r$ induces the identity on $H^1(S^1\times \Sigma;\Z_2)$, the double is spin in this case as well.
\end{proof}

Before showing that the $\Z_2$--construction is spin, we give $H_2$ generators of the mapping torus of $\Sigma$ with monodromy $r$, which is a submanifold of our $\Z_2$--construction.

\begin{proposition}\label{prop: H_2(MT) generators}
	Let $r$ be the reflection map on a closed orientable genus $g$ surface $\Sigma$ from the proof of Proposition \ref{prop:spin_preserving_map}, and let $\{x_1,y_1,\cdots x_g,y_g\}$ be the standard symplectic generators of $H_1(\Sigma;\Z)$. Let $Y$ be the mapping torus $Y=MT(\Sigma,r)$, which is a fibration over $S^1$ with $\Sigma$ fibers. Then $H_2(Y;\Z_2)$ is generated by the following:\begin{itemize}
		\item The surface $\Sigma$, which is a fiber of the mapping torus.
		
		\item A collection of $g$ embedded tori which are obtained by taking the parallel transport of each $x_i\subset \Sigma$ along the $S^1$ base of the fibration on $Y$.
		
		\item A collection of $g$ embedded Klein bottles which are obtained by taking the parallel transport of each $y_i\subset \Sigma$ along the $S^1$ base of the fibration on $Y$.
	\end{itemize}
\end{proposition}

\begin{proof}
	Since the map $r_*$ induces the identity on $H_*(\Sigma;\Z_2)$, the long exact sequence for the homology of a mapping torus shows that $H_2(Y;\Z_2)\cong H_2(\Sigma;\Z_2)\oplus H_1(\Sigma;\Z_2)$ \cite[Example 2.48]{Hatcher_AT}. This implies that the rank of $H_2(Y;\Z_2)$ is $2g+1$. We will describe its $2g+1$ elements and then we will argue that they are linearly independent. The element coming from $H_2(\Sigma;\Z_2)$ is the class $[\Sigma]$ of a regular fiber. The elements corresponding to the $H_1(\Sigma;\Z_2)$ summand are the parallel transports of the generators of $H_1(\Sigma;\Z_2)$ in the mapping torus, which are the $g$ pairs of tori and Klein bottles in the proposition statement. In other words, they are images of $I \times x_i$ and $I\times y_i$ under $q: I \times \Sigma_g \to Y$ .

	To see that they are independent let
	$$a_0[\Sigma] + \sum b_i [q(I\times x_i)] + \sum c_i [q(I \times y_i)] = 0.$$
	
	Take a cup product of both sides of the equation with $[q(I\times \set{pt})]$. By intersection theory this gives $a_0$ on the left-hand side so $a_0= 0$. Similarly multiply with $[q(\set{pt} \times x_i)]$ and $[q(\set{pt} \times y_i)]$ to get $c_i= 0$ and $b_i=0$.
\end{proof}
	
\subsection{Intersection form of $\mathbf{\Z_2}$--construction} \label{sec:Z2_spin}

Equipped with the formula $w_2(TM)\cdot x = x \cdot x$ for any $x\in H_2(M;\Z_2)$, we can detect whether $M$ is spin if we have generators of $H_2(M;\Z_2)$ and their squares. The typical manifold that we will start with will be a $\Z_2$--construction of a $4$--manifold $X$ along a surface $\Sigma \subset X$, where $X$ is spin and $X-\Sigma$ is simply-connected. We will denote the result of the $\Z_2$--construction by $Z$. The gluing map in the $\Z_2$--construction is assumed to be $a\times r$, where $a:S^1\to S^1$ is the antipodal map and $r$ is the orientation--reversing diffeomorphism of $\Sigma$ described in Proposition \ref{prop:spin_preserving_map}. Proposition \ref{prop:spin_preserving_map} shows that the universal cover of $Z$ is spin, and we next show that $Z$ is also spin. It will be useful to think of $Z$ as $(X-\nu \Sigma_2)/(a\times r)$, where we are quotienting by $a\times r$ on $\partial(X-\nu\Sigma)$. Let $MT$ be a mapping torus of $\Sigma$ with monodromy $r$. Then $Z$ has an embedded copy of $MT$, which is the image of $\partial(X-\nu\Sigma)$ after quotienting by $a\times r$. Its neighborhood $\nu MT\subset Z$ is a $D^1$--bundle over $MT$ whose total space is orientable. In particular, $\partial(\nu MT)$ is a copy of the orientation double cover of $MT$, $\Sigma_g \times S^1$. Taking this into consideration, we can decompose $Z$ as $Z= (X-\Sigma) \cup_{S^1\times \Sigma} \nu MT$.

\begin{proposition}\label{prop:Z_H2_generators}
	Let $X$ be a smooth $4$--manifold with an embedded square zero orientable surface $\Sigma$, such that $X-\nu \Sigma$ is simply--connected. Let $Z$ be the $\Z_2$--construction of $X$ along $\Sigma$. With the decomposition $Z= (X-\Sigma) \cup_{S^1\times \Sigma} \nu MT$ as described above, $H_2(Z;\Z_2)$ is generated by the following classes:
	
	\begin{itemize}
		\item[(i)]  Elements in $H_2(Z;\Z_2)$ represented by surfaces in $X-\Sigma \subset Z$.
		
		\item[(ii)] The collection of $g$ embedded tori and $g$ embedded Klein bottles, which are generators of $H_2(MT;\Z_2)$ (c.f. Proposition \ref{prop: H_2(MT) generators}).
		
		\item [(iii)] One surface $S_{\mu}\subset Z$, where $S_\mu \cap \partial \nu MT=S^1\times \{pt\}$, where we identify $\partial \nu MT=\partial (X-\nu \Sigma)$ with $S^1\times \Sigma$. 
	\end{itemize}
\end{proposition}

\begin{proof}
	In this proof, we will always use $\Z_2$ coefficients for homology, and omit them for brevity. First, apply Mayer-Vietoris to $Z = (X- \nu \Sigma) \cup (\nu MT)$.
	
	$$\cdots H_2(X-\nu \Sigma)\oplus H_2(MT)\to H_2(Z)\to H_1(S^1 \times \Sigma) \to 0 \oplus H_1(MT) \to H_1(Z)\to 0. $$
	
	It's easily checked that $H_1(MT)= \Z_2^{2g} \oplus \Z_2$, with generators $c_i$ coming from the fiber $\Sigma$ and $b$, the base circle. The map $H_1(S^1\times \Sigma)\to H_1(MT)$ is induced by the inclusion $\partial(\nu MT)=S^1\times \Sigma \hookrightarrow \nu MT$, which sends $[S^1\times \{pt\}]$ to $2b=0$ and is the identity on the generators $c_i$ of $\Sigma$ (which we may also regard as generators of $H_1(Z)$ and $H_1(S^1\times \Sigma)$). Therefore the image of the map $H_1(S^1\times \Sigma)\to H_1(MT)$ is isomorphic to $H_1(\Sigma)$, so the long exact sequence becomes
	
	$$\cdots H_2(X-\nu \Sigma)\oplus H_2(MT)\to H_2(Z)\to H_1(S^1 \times \Sigma) \to H_1(\Sigma)\to 0.$$
	
	We may deduce further that the image of $H_2(Z) \to H_1(S^1 \times \Sigma)$  is the kernel of the aforementioned $H_1(S^1\times \Sigma)\to H_1(MT)$, namely the meridian generator $S^1\times \set{pt} \subset S^1 \times \Sigma = \partial (X-\nu \Sigma)$.
	So we get another exact sequence:
	$$\cdots H_2(X-\nu \Sigma_2)\oplus H_2(MT)\to H_2(Z)\to H_1(\mu_{\Sigma})\to 0.$$
	
	So $H_2(Z)$ is generated from the right side by a class mapping to $\mu_{\Sigma}$. Considering the geometric meaning of the boundary homomorphism $H_2(Z)\to H_1(S^1\times \Sigma_g)$, we see that the class mapping to $\mu_{\Sigma}$ is precisely the class represented by $S_\mu$ in the proposition statement. From the left side of the sequence, we see that $H_2(Z)$ is also generated by the images of generators of $H_2(X-\Sigma)$ and $H_2(MT)$. Thus we've found a complete generating set of $H_2(Z)$. It remains to mention that we can skip the generator represented by the surface $\Sigma$ coming from Proposition \ref{prop: H_2(MT) generators}. The surface $\Sigma \subset MT$ can be pushed-off to $X-\Sigma$, so this class is covered by (i).
	
\end{proof}

\begin{lemma} 
	There exists a Möbius band $W\subset \nu MT$ such that $\partial W = \mu_{\Sigma} \subset \partial(\nu MT)$, and a parallel copy of $W$ is a disjoint Möbius band.
	\label{lem:W_mb}
\end{lemma}

\begin{proof}
	$\nu(MT)$ is a non-orientable $D^1$-bundle over $MT$, and $MT$ is a non-orientable $\Sigma$ bundle over $S^1$. This gives us the following stack of fiber--bundle diagrams:\
	\begin{center}
		\begin{tikzcd}
			I \arrow[r,hook] & \nu MT \arrow[d,"g"] \\
			\Sigma_2 \arrow[r, hook] & MT \arrow[d, "f"] \\		
			& S^1
		\end{tikzcd}
	\end{center} Let $p\in \Sigma_g$ be a fixed point of the reflection $r$ (note that there are many such fixed points). Then the parallel transport of $p$ in the fibration $MT\to S^1$ is a copy of $S^1$, denoted $\tilde S$. Its preimage $g^{-1}(\tilde S)$ is a non-orientable $D^1$--bundle over $\tilde S$, and so is a Möbius band in $\nu MT$. Its boundary is a copy of $S^1\times \{p\}\subset S^1\times \Sigma_g$. Note that if we take another fixed point $p'$ close to $p$, we get a disjoint push-off Möbius band whose boundary is $S^1\times \{p'\}$.
	
\end{proof}

\begin{proof} [Proof of Theorem \ref{thm:Z2_spin}]
	Recall from Proposition \ref{prop:w_2_nonorientable_surface} that if each generator of $H_2(Z;\Z_2)$ has even self--intersection, then $w_2(Z)\equiv 0$, and so $Z$ is spin. We'll continue to shorten $H_2(Z;\Z_2)$ to $H_2(Z)$ in this proof. Consider the generators of $H_2(Z)$ given in Proposition \ref{prop:Z_H2_generators}. Since $X$ is spin, elements in $(i)$ are represented by orientable surfaces and have even square. For classes of type $(ii)$, one can take a parallel copy of each $c_i$ generator of $H_1(\Sigma_g)$ to get a disjoint push-off of a torus or Klein bottle. Hence, the only surface which may have odd self--intersection is $S_\mu$.
	
	By Mayer-Vietoris, $S_{\mu}$ can be represented as a union of $S_1$ and $S_2$ with $[S_1] \in H_2(X-\nu \Sigma_2, \partial)$ and $[S_2] \in H_2(\nu MT, \partial)$, so that $\partial S_1= \mu = \partial S_2$. We may use the Möbius band $W$ from Lemma \ref{lem:W_mb} as $S_2$. Let $W'$ be its push-off. Since $H_1(X-\Sigma_g)=0$, $\partial W$ and $\partial W'$ both bound surfaces in $X-\Sigma_g$, which we denote $S$ and $S'$, respectively. By considering the inclusion $X-\Sigma_g\subset X$, these two surfaces can be thought of as (non-closed) surfaces in $X$, which we can cap off with the disjoint disks $D^2\times \{p\}$ and $D^2\times \{p'\}$ to get closed surfaces $\bar S,$ $\bar S'\subset X$. Since $X$ is spin, $\bar S\cdot \bar S'\equiv 0\bmod 2$, and since $D^2\times \{p\} \cap D^2\times \{p'\}=\varnothing$, we can assume all intersection points are away from these two disks. Hence the surfaces $S$ and $S'$ also have an even number of intersection points. Then the surfaces $S\cup_{S^1\times \{p\}} W$ and $S'\cup_{S^1\times\{p'\}} W'$ are closed surfaces in $Z$, each intersecting $\partial (\nu MT)$ in a meridian. It follows that $[S\cup W]=S_\mu$, and that $[S'\cup W']$ is its push-off. Hence $S_\mu$ has even self-intersection, and so $Z$ is spin.
\end{proof}

\section{$\Z_2$--geography for $w_2$--type (ii).}\label{sec:type2}

This section aims to find smooth, spin, irreducible manifolds with order two $\pi_1$. The universal cover of such manifolds will also be spin, and so these are manifolds of $w_2$--type $(ii)$. Such manifolds will have an intersection form of $n(-E_8)\oplus \ell H$ for some even $n$. Under the condition $\ell \geq 2n-1$, these manifolds will be homeomorphic to $L_2\# E(n) \# k(S^2\times S^2)$, where $k= \ell + 1 - 2n$. Let $R_{n,k}:=L_2\# E(n) \# k(S^2\times S^2)$. Our goal here is to find irreducible copies of $R_{n,k}$ for each $(n,k)$--coordinate with $n$ even and $k\geq 0$.

\subsection{Even $\mathbf{b_2^+}$}\label{sec:type2_even} In this subsection, we find spin irreducible copies of $R_{n,k}=L_2\# E(n) \# k(S^2\times S^2)$ with even $b_2^+$. When $n>0$, $R_{n,k}$ has an intersection form of $n(-E_8)\oplus (k+2n-1)H$, and has algebraic invariants $e(R_{n,k})=2k+12n$, $\sigma(R_{n,k}) = -8n$, and $b_2^+ (R_{n,k}) = k+2n-1$. When $n=0$, $R_{n,k}$ has $kH$ as its intersection form, $e=2+2k$, $\sigma =0$, and $b_2^+ = k$. From this we deduce that this subsection populates the $(n,k)$ integral lattice spanned by $(n,k)$--coordinates where $n$ is strictly positive and even, and $k$ is odd, along with the region spanned by $(0,k)$ for $k$ even. More formally, we aim to realize the following coordinates:
\[
	\{(n,k)\in \Z_{>0}\times \Z_{>0} : n \text{ is even, } k \text{ is odd}\} \cup \{(0,k) \in \{0\} \times \Z_{>0}: k\text{ is even}\}.
\]

\subsubsection{$R_{2s+4, 2n+1}$ for $n,s\geq 0$ ($\sigma \leq  -32$)} \label{sec:sigma<-32}
Park and Szabó in \cite{Park_Szabo} fill a simply-connected, irreducible, spin geography. We will reintroduce their main construction for the convenience of the reader and follow the notation of the authors.

\begin{remark}[Construction of $M_n(s)$] \label{rmk:Mn(s)_construction}
	Let $Y$ be the mapping torus on $T^2$ with monodromy $\phi(a,b) = (ab,b)$. Define $Z=S^1\times Y$, and $W= Z\#_{T_2} Z$ a fiber sum along certain tori in $Z$. Although $W$ has a square zero genus two surface $\Sigma_W$, it's not obvious if its complement is simply-connected. Next, consider the elliptic surface $E(2)$. A \emph{nucleus} of $E(2)$ is a regular neighborhood of a cusp fiber along with a section which transversely intersects it. It's shown in \cite{Gompf-Mrowka} that $E(2)$ contains three disjoint nuclei, and that we can arrange for the regular fiber and section in  the first nucleus and the regular fiber from the second nucleus to all be symplectic. Let $N_1$ and $N_2$ denote the two nuclei, and let $S_i$ and $T_i$ be the embedded spheres and tori in each $N_i$. The knot surgered $E(2)$, denoted $E(2)_K$, is obtained by performing Fintushel-Stern knot surgery on $E(2)$ along $T_1$ \cite{FS_knots}. Assuming the knot $K$ is fibered and has a genus one, such as the trefoil, the surgered manifold $E(2)_K$ has a square zero genus two surface $\Sigma_K$, which is roughly obtained by tubing together $T_1$, $S_1$, and a Seifert surface of $K$. Since $T_2$ is disjoint from $N_1$, we may take a fiber sum of $E(2)_K$ and $E(2s)$ along $T_2$ and a regular fiber of $E(2s)$. Let $M(s)$ denote this fiber sum. Finally, the manifold $M_n(s)$ is a fiber sum of of $n$ copies $W$ along $\Sigma_W$ and a single copy of $M(s)$ along $n$ parallel copies of $\Sigma_K$. They show in \cite{Park_Szabo} that $\pi_1(M_n(S))=1$ and that the invariants are $\sigma(M_n(s))= -16s-16$ and $e(M_n(s))= 24s+4n+24$.
\end{remark}

If we had a symplectic $\Sigma_2\subset M_n(s)$ with a trivial meridian, we could apply the $\Z_2$--construction on $M_n(s)$ to fill a geography with fundamental group of order two. We are not convinced that such a genus two embedded surface exists, so we will modify $M_n(s)$ to get around this. Let $E(2)_{K'}$ be another copy of $E(2)_K$ with $K=K'$ and nuclei $N'_1$, $N'_2$ containing regular fibers $T'_i$ and dual spheres $S'_i$, with the knot surgery done on $T'_1$. Define $M'_n(s)= \Fibersumd{M_n(s)}{E(2)_{K'}}{T_2 = T'_2}$. Same as with $E(2)_K$, we have genus 2 surface in $E(2)_{K'}$ of self-intersection 0 given by resolving the union of $T'_1$ and $S'_1$, and capping with a Seifert surface of $K'$. Call the surface $\Sigma'_2$. Now $M'_n(s) - \Sigma'_2$ contains a disk in $N'_1- \Sigma'_2$ bounding the meridian of $S'_1$, which is the same as a meridian of $\Sigma'_2$. To see this, take another regular fiber in $N'_1$ and cap its generators using the cusp fiber of $N'_1$. This implies that $M'_n(s) - \Sigma'_2$ is simply-connected. Moreover, $M'_n(s)$ is spin and symplectic because $M_n(s)$ and $E(2)_K$ are \cite{GompfNew}. So apply the $\Z_2$--construction to $M'_n(s)$ along $\Sigma'_2$ and call it $U_n(s)$. By Theorem \ref{thm:Z2_spin}, $U_n(s)$ is spin and by Lemma \ref{lem:symplectic_irreducible} it and its universal cover are irreducible.

Now let's compute the invariants. $M_n(s)$ is parametrized by $n,s\geq 0$ and has invariants $\sigma(M_n(s))= -16s-16$ and $e(M_n(s))= 24s+4n+24$.   Then $M'_n(s)$ has $\sigma(M'_n(s)) = -16s-32$ and $e(M'_n(s))= 24s+4n+48$. The intersection form of $U_n(s)$ is $(2s+4)E_8 \oplus (4s+2n+8)H$. So $U_n(s)$ is an irreducible copy of $L_2 \# E(2s+4) \# (2n+1)S^2 \times S^2$.

\subsubsection{$R_{0,2n+2}$ for $n\geq 5$ ($\sigma = 0$ line).} \label{sec:sigma_0}  In \cite[Section 4.2]{baykurhamadaspin}, the following positive factorization is given in $\Mod(\Sigma_g^1)$ for $g\geq 6$. \[
	t_{d_2}t_dt_{d_1}t_{\tilde d}[\psi,t_{d_3}^{-1}\phi_1][1,1] = t_{\delta}^0.
\] The exact curves and maps are not important for our argument, and the reader can consult \cite{baykurhamadaspin} for more details. The important thing for us is that this positive factorization prescribes a genus $g=n+1$ Lefschetz fibration $\pi_n:\mathcal Z_n\to \Sigma_2$ whose total space is spin and admits a square zero section $\Sigma$. Baykur and Hamada have shown that by performing Luttinger surgeries on a link of Lagrangian tori in $\mathcal Z_n$, one obtains a manifold $Z_n$ which is an irreducible copy of $\#_{2n+1} S^2\times S^2$ for all $n\geq 5$. The section $\Sigma$ is disjoint from the link of Lagrangian tori, and so survives as a square zero genus two surface in $Z_n$. We claim that the $\Z_2$--construction of $Z_n$ along $\Sigma$ is an irreducible copy of $R_{0,2n+2}$. Once we show that $Z_n-\Sigma$ is simply--connected, this will follow from the fact that the $\Z_2$--construction preserves the spin structure on $Z_n$ (see Section \ref{sec:Z2_spin}).

\begin{figure}
	\centering
	\resizebox{5in}{!}{
	\begin{tikzpicture}
		
		
		\draw[thick] (8,8) -- (-2,8) -- (-2,-2) -- (8,-2) -- (8,8);
		
		\draw[thick] (-2,5.5) -- (8,5.5);
		
		\draw[thick](3,5.5) -- (3,-2);
		
		\fill[color=BurntOrange!30!white] (-2,1) -- (3,1) -- (3,-2) -- (-2,-2) -- (-2,1) (.5,-.5) node[scale=2, color=BurntOrange]{$\mathcal K$};
		\draw[thick,color=BurntOrange] (-2,1) -- (3,1);
		
		\draw[ultra thick, color =TealBlue] (-2,5.5) -- (3,5.5) -- (3,-2) -- (-2,-2) -- (-2,5.5);
		\draw[color=TealBlue] (-.5,2) node[scale=2] {$A$};
		\draw [->,color=TealBlue] (0,2) .. controls + (.5,0) and +(0,0) .. (1,3);
		\draw [->,color=TealBlue] (-.5,1.5) .. controls + (0,-.5) and +(0,0) .. (0,0);
		
		\draw (3,6.75) node [scale=2] {$B$};
		
		\draw (5.5,1.75) node[scale=2] {$C$};

		
		\draw[thick] (-7,6) -- (-7,7) .. controls +(0,1) and +(-1,0) ..(-6,8) .. controls +(1,0) and + (0,1).. (-5,7) -- (-5,6); 
		
		\draw[thick] (-6,7) circle [radius = .25];
		
		\draw[thick, dashed,color=BrickRed] (-7,6) .. controls +(0,.35) and +(0,.35).. (-5,6);
		\draw[thick,color=BrickRed] (-7,6) .. controls +(0,-.35) and +(0,-.35).. (-5,6);
		\fill[color=BrickRed] (-5.3,5.8) circle [radius=2pt];
		
		\draw (-6,6) node[color=BrickRed] {$p$};
		\draw[color=BrickRed] (-5.85,6) -- (-5.5,5.9);
		\draw[color=BrickRed] (-5.9,5.9) -- (-5.45,4.9);
		
		
		\draw[thick] (-7,5) -- (-7,-1) .. controls +(0,-1) and +(-1,0) .. (-6,-2) .. controls +(1,0) and +(0,-1) .. (-5,-1) -- (-5,5) ;
		
		\draw[thick,color=BrickRed] (-7,5) .. controls +(0,.35) and +(0,.35).. (-5,5);
		\draw[thick,color=BrickRed] (-7,5) .. controls +(0,-.35) and +(0,-.35).. (-5,5);
		\fill[color=BrickRed] (-5.3,4.8) circle [radius=2pt];
		
		\draw[thick] (-6,-1) circle [radius = .25];
		\draw[thick] (-6,.5) circle [radius = .25];
		\draw[thick] (-6,3.5) circle [radius = .25];
		
		\draw(-6, 2.25) node[scale=2.5] {$\vdots$};
		
		
		\begin{scope}[yshift=1cm]
		
		\draw[thick] (8,-5) .. controls + (0,1) and + (1,0) ..  (7,-4) -- (-1,-4) .. controls  + (-1,0) and + (0,1) ..  (-2,-5) .. controls +(0,-1) and +(-1,0) ..  (-1,-6) -- (7,-6) .. controls + (1,0) and +(0,-1) .. (8,-5);
		
		\draw[thick] (5.5,-5) ellipse (.5 and .3);
		
		\draw[thick] (.5,-5) ellipse (.5 and .3);
		
		\draw[thick,color=ForestGreen] (3,-4) .. controls + (.35,0) and +(.35,0) .. (3,-6);
		
		\draw[thick,color=ForestGreen,dashed] (3,-4) .. controls + (-.35,0) and +(-.35,0) .. (3,-6);
		
		\fill[color=ForestGreen] (3.25,-5) circle [radius = 3pt] (3.35,-5) node [anchor = west] {$q$};
		
		\end{scope}
		
		
		\draw [decorate, decoration = {brace, mirror, amplitude=10pt}, thick] (-2.5,5.4) --  (-2.5,1.2) (-2.8,3.25) node[anchor = east,scale=1.2] {$\Sigma_{g-6}^2$};
		
		\draw [decorate, decoration = {brace, mirror, amplitude=10pt}, thick] (-2.5,.8) -- (-2.5,-2)(-2.8,-.6) node[anchor = east,scale=1.2] {$\Sigma_{5}^1$};

		\draw [decorate, decoration = {brace, amplitude=10pt}, thick] (8.5,5.4) --  (8.5,-2) (8.8,1.65) node[anchor = west,scale=1.2] {$\Sigma_{g-1}^1$};
		
		\draw [decorate, decoration = {brace, mirror, amplitude=10pt}, thick] (8.5,5.6) --  (8.5,7.9) (8.8,6.75) node[anchor = west,scale=1.2] {$\Sigma_{1}^1$};

	\end{tikzpicture}
	
	}
	\caption{}
	\label{fig:A_B_C}	
\end{figure}

To see that $\pi_1(Z_n-\Sigma)=1$, we adopt the decomposition of $\mathcal Z_n$ from \cite{baykurhamadaspin}. We can decompose $\mathcal Z_n$ as \begin{equation}\label{eq:Zn_decomp}
	\mathcal Z_n = (\mathcal K \cup (\Sigma_{g-6}^2\times \Sigma_1^1))\cup (\Sigma_1^1\times \Sigma_2) \cup (\Sigma_{g-1}^1\times \Sigma_1^1),
\end{equation} where $\mathcal K$ contains all the Lefschetz singularities, and outside of $\mathcal K$, the map $\pi_n$ is given by projection onto the second factor. Let $A$, $B$, and $C$ denote the first, second, and third pieces of this decomposition. We can assume the section $\Sigma$ lies in $\{p\}\times \Sigma_2\subset \partial(\Sigma_1^1)\times \Sigma_2\subset B$. We fix a point $q\subset \Sigma_1^1\subset \Sigma_2$, so that $(p,q)$ is a base point for our $\pi_1$ computations. Figure \ref{fig:A_B_C} shows a schematic picture of this decomposition for $\mathcal Z_n$.

Let $\{a_1,b_1,\cdots a_g,b_g\}$ be the standard generators of $\pi_1(\Sigma_g)$ and let $\{x_1,y_1,x_2,y_2\}$ generate $\pi_1(\Sigma_2)$. Without loss of generality, we will consider these curves as $\pi_1$ generators for subsurfaces such as $\Sigma_1^1\subset \Sigma_2$. By \cite[Proposition 2b]{baykurhamadaspin}, there's a sequence of Luttinger surgeries on $C$ such that the resulting manifold $C'$ has a fundamental group normally generated by two disjoint, pairwise non-separating curves $a,b\subset \Sigma_{g-1}^1\times \{q\}$. Similarly, by \cite[Proposition 2a]{baykurhamadaspin}, there is a sequence of Luttinger surgeries on $B$ such that the resulting manifold $B'$ has a fundamental group normally generated by the curves $x_2,y_2\subset \{p\}\times \Sigma_2$. Regarding $\Sigma$ as a subsurface of $B'$ which survives the Luttinger surgeries, $\pi_1(B'-\Sigma)$ is normally generated by $x_2, y_2$, and $\mu_{\Sigma}$, where the latter is a word in $a_g, b_g,$ and $\partial(\Sigma_1^1)\times \{q\}$. We claim that $\pi_1(B'-\Sigma)$ is normally generated by $\{x_2,y_2,\partial(\Sigma_1^1)\times \{q\}\}$, i.e. that $a_g$ and $b_g$ are not needed for normal generation after removing $\Sigma$. To see this, recall that since $B'$ is normally generated by $x_2$ and $y_2$, there exist immersed annuli $H_{a_g},H_{b_g}\subset B'$ which are homotopies from the loops $a_g$ and $b_g$ to conjugates of concatenations of $x_2$ and $y_2$. Removing $\Sigma$ from $B'$ is equivalent to adding instead the product of a $1$--handle and a copy of $\Sigma_2$ to $B'$ along its boundary, so we may regard $B'$ as a subspace of $B'-\Sigma$. Hence $H_{a_g}$ and $H_{b_g}$ survive in $B'-\Sigma$ as well, and so the claim holds.

When we glue $B'-\Sigma$ and $C'$ together along $\partial(\Sigma_1^1)\times \Sigma_1^1$, the generators corresponding to $x_2,y_2\subset B'$ are identified with curves in $\{p\}\times \Sigma_1^1$, and so may be (normally) expressed in terms of $a$ and $b$. The boundary curve $\partial (\Sigma_1^1)\times \{q\}$ is identified with $[a_1,b_1]\cdots [a_{g-1},b_{g-1}]\subset \Sigma_{g-1}^1\times \{q\}$, and so is also, up to conjugation, a word in $a$ and $b$. Hence, $\pi_1((B'-\Sigma)\cup C')$ is normally generated by $\{a,b\}$, just as $B'\cup C'$ is. Therefore the rest of the Seifert-van Kampen theorem arguments from \cite{baykurhamadaspin} used to show that $\pi_1(Z_n)=1$ will also apply for $\pi_1(Z_n-\Sigma)$.

\subsubsection{$R_{2m,2n+3}$ for $m\geq 1$, $n\geq 5$.} \label{sec:sigma_0_generalization}

Before we start, the following lemma will be of use:

\begin{lemma} \label{lemma: normal_generation}
	Suppose $\pi_1(X)$ is generated by $g_i$ for $1\leq i \leq n$. Let $T_j$ and $T'_j$  for $1 \leq j \leq k$ be tori in $X$ such that $T_j$ are pairwise disjoint, and $T_j$ intersects $T'_j$ in a single point for all $j$. In addition, say $\pi_1(T'_j)$ can be generated by loops representing $g_i$, up to conjugation by paths $\gamma_j$ connecting the basepoint of $X$ and the basepoint of $T'_j$. Also assume that after removing $T:= \sqcup_j T_j$ from $X$, each $T'_j$ gets punctured only once. Then $\pi_1(X-T)$ is also normally generated by $g_i$.
\end{lemma}

\begin{proof}
	Since $\pi_1(X)$ is generated by $g_i$, note that $\pi_1(X-T)$ is normally generated by $g_i$ and meridians $\mu_j$ of $T_j$ (conjugated by $\gamma_j$). On the other hand, if $\pi_1(T'_j)$ is generated by $x,y$, then $[x,y]$ is a meridian of $T_j$. By assumption, $x$ and $y$ are generated by $g_i$ up to conjugation, and so is $[x,y]$. So by quotienting $\pi_1(X-T)$ by $g_i$ we also quotient all $x,y$ and all conjugates of $[x,y]$. So $\mu_j$ will be trivial in $\pi_1(X-T)/N(g_1,\dots, g_n)$ and so $\pi_1(X-T)/N(g_1,\dots, g_n)= \pi_1(X-T)/N(g_1,\dots, g_n,\mu_1,\dots, \mu_k ) = 1$.

\end{proof}

We demonstrate how to modify the $\sigma = 0$ construction by fiber-summing with a copy of $E(2n)$, which gives a large region of the geography plane. Recall from Section \ref{sec:sigma_0} that in \cite[Section 4.2]{baykurhamadaspin}, for $n\geq 5$ they construct a spin 4--manifold $\mathcal Z_n$ which admits a genus $g=n+1$ Lefschetz fibration $\pi_n$ over $\Sigma_2$ with a square zero section $\Sigma$. A sequence of Luttinger surgeries on $\mathcal Z_n$ gives an irreducible copy of $\#(2n+1)S^2\times S^2$. The fundamental group of $\mathcal Z_n$ is generated by the standard generators of $\pi_1(\Sigma_2)$ and $\pi_1(\Sigma_g)$, which we denote by $\{x_1,y_1,x_2,y_2\}$ and $\{a_1,b_1,\dots, a_g,b_g\}$, respectively. Recall the decomposition of $\mathcal Z_n$ provided in (\ref{eq:Zn_decomp}), where $A$, $B$, and $C$ denote the first, second, and third blocks in the decomposition. In what follows, we will often be looking at codimension $0$ submanifolds of $\mathcal Z_n$ which are diffeomorphic to $U\times V$, where $U$ is a surface with boundary embedded in $\Sigma_g$ and $V$ is a a surface with boundary embedded in $\Sigma_2$, where $\pi_n$ is given by projection onto $V$. We will also be looking at codimension $1$ submanifolds of the form $\partial U\times V$ and $U \times \partial V$. We will also often be decomposing the base and fiber as $\Sigma_2 = \Sigma_1^1\cup \Sigma_1^1$ and $\Sigma_g = \Sigma_{g-1}^1\cup \Sigma_1^1$, respectively. Note that unlike $B$ and $C$, $A$ is not a product space, as it contains all the Lefschetz singularities of $\pi_n$. That said, its boundary is a product space given by $(\partial \Sigma_{g-1}^1)\times \Sigma_1^1\cup_{(\partial\Sigma_{g-1}^1) \times \partial\Sigma_1^1} \Sigma_{g-1}^1 \times \partial\Sigma_1^1$. With all this established, we provide the gluing regions between $A$, $B$, and $C$.\begin{itemize}
	\item $C$ and $B$ are glued along $\partial \Sigma_{g-1}^1\times \Sigma_1^1\subset \partial C$ and $\partial \Sigma_1^1 \times \Sigma_1^1 \subset \partial B$.
	
	\item $C$ and $A$ are glued along $\Sigma_{g-1}^1\times \partial \Sigma_1^1$.
	
	\item $A$ and $B$ are glued along $(\partial\Sigma_{g-1}^1)\times \Sigma_1^1\subset \partial A$ and $(\partial\Sigma_1^1)\times \Sigma_1^1\subset \partial B$.
\end{itemize} We pick a point $p$ in $\partial\Sigma_{g-1}^1\subset \Sigma_g$ and $q\in \partial\Sigma_1^1\subset \Sigma_2$, so that $(p,q)$ may be a basepoint for $\pi_1$ computations of all three blocks $A$, $B$, and $C$. Figure \ref{fig:A_B_C} shows a labeled schematic picture of this decomposition.

\begin{figure}
	\centering
	\resizebox{!}{1.75in}{
	\begin{tikzpicture}
		
		\begin{scope}[scale=.8]
			
			

			\draw[thick, color=blue, name path = b] (0,0) circle [radius = 1.25cm];
			\draw [->, color = blue] (1.8,-1.8) -- (1,-1); 
			\draw[thick, color = blue] (1.7,-1.7) node[anchor=north west, scale = .8]{$a''_g$};
			
			\draw[thick, color=blue, name path = b] (0,0) circle [radius = 2cm] (-2,0) node[anchor=east] {$a'_g$};
			
			\draw[thick, color = blue, name path = a] (0,.75) .. controls +(-.25,0) and +(-.25,0) .. (0,3); 
			\draw[thick, color = blue, dashed] (0,.75) .. controls +(.25,0) and +(.25,0) .. (0,3) (0,3) node[anchor= south] {$b'_g$};

			\draw[very thick] (3,3) -- (0,3) .. controls +(-2,0) and + (0,2) .. (-3,0).. controls +(0,-2) and +(-2,0) .. (0,-3) -- (3,-3);
			\draw[very thick] (3,0) ellipse (.65 and 3);
			\draw[very thick] (0,0) circle [radius = .75cm];
			
		\end{scope}

		\begin{scope}[xshift=9cm, scale = 1.1]
			
			\draw[thick] (-4.25,0) node[scale = 2] {$\times$};
			
			\draw[thick, color=blue, name path = b1] (-1.25,0) circle [radius = 1cm](-2.2,0) node[anchor=east]{$x'_1$};
			\draw[thick, color=blue, name path = b2] (1.25,0) circle [radius = 1cm] (2.2,0) node[anchor=west] {$x'_2$};
			
			\draw[thick,color=blue, name path = a1] (-1.25,.5) .. controls +(.25,0) and +(.25,0) .. (-1.25,2)(-1.25,2) node [anchor = south] {$y'_1$};
			\draw[thick,color=blue, dashed] (-1.25,.5) .. controls +(-.25,0) and +(-.25,0) .. (-1.25,2);

			\draw[thick, color=blue] (1.25,0) circle [radius = .7cm] (.75,0);
			\draw [->, color = blue] (2,-1) -- (1.7,-.6); 
			\draw[thick, color = blue] (1.8,-.95) node[anchor=north west, scale = .8]{$x''_2$};
			
			\draw[thick, color=blue] (-1.25,0) circle [radius = .7cm] (.75,0);
			\draw [->, color = blue] (-2,-1) -- (-1.7,-.6); 
			\draw[thick, color = blue] (-1.8,-.95) node[anchor=north east, scale = .8]{$x''_1$};
			
			\draw[very thick] (1,2) -- (-1,2) .. controls +(-1,0) and + (0,1) .. (-3,0).. controls +(0,-1) and +(-1,0) .. (-1,-2) -- (1,-2) .. controls +(1,0) and +(0,-1) .. (3,0) .. controls +(0,1) and +(1,0) .. (1,2);
			
			\draw[very thick] (-1.25,0) circle[radius = .5cm];
			\draw[very thick] (1.25	,0) circle[radius = .5cm];
			
		\end{scope}
		
	\end{tikzpicture}	
	}
	\caption{}
	\label{fig:B_lag_tori}	
\end{figure}

Next, we perform Luttinger surgeries on $B$, in a similar style to \cite[Proposition 2(a)]{baykurhamadaspin}. $\pi_1(B,(p,q))$ is generated by $a_g,b_g$, which come from generators of $\pi_1(\Sigma_g)$, and $x_1,y_1,x_2,y_2$, which come from generators of $\Sigma_2$. For any $\pi_1$ generator $c$, let $c'$ and $c''$ denote push-offs of curves freely homotopic to $c$, which are pairwise disjoint. Define four disjoint Lagrangian tori in $B$ as follows:\[
	H_1 = a'_g\times x_2', H_2 = b'_g\times x_2'', H_3 = a'_g\times x_1', H_4 = a''_g\times y_1'.
\] Figure \ref{fig:B_lag_tori} shows the relevant curves on $\Sigma_1^1\subset \Sigma_g$ and $\Sigma_2$. We perform three Luttinger surgeries:\[
	(H_1,a'_g,1), (H_2, b_g',1), (H_3, x_1',1)
\] and denote the resulting manifold $B'$. Lemma \ref{lemma: normal_generation} shows that $\pi_1(B')$ is normally generated by the generators of $\pi_1(B)$. Moreover, the following relations hold:\begin{equation}\label{eq:L_surg_relations}
	\{b_g,y_2\}a_g = \{a_g,y_2\}b_g = \{b_g,y_1\}x_1 =1,
\end{equation} where $\{c,d\}$ denotes a commutator of conjugates of $c$ and $d$. The fourth torus $H_4$ survives the Luttinger surgeries, and has a geometric dual $b_g'\times x_1''$, which also survives the surgeries. Hence $H_4$ is a homologically essential Lagrangian torus in $B'$. With this established, we take a symplectic fiber sum along $H_4$. Set $B_m = B' \#_{H_4=F} E(2m)$, where $F$ is a regular fiber of $E(2m)$ and we choose a gluing diffeomorphism so that the resulting manifold is spin \cite[Proposition 1.2]{GompfNew}. A simple exercise using the Seifert-van Kampen theorem shows that $\pi_1(B_m)$ has the same normal generating set as $\pi_1(B')$ with the added relations $a_g=y_1=1$. But if $a_g=y_1=1$, the torus surgery relations in (\ref{eq:L_surg_relations}) then imply $x_1=b_g=1$ as well. Therefore, $\pi_1(B_m)$ is normally generated by $x_2$ and $y_2$.

Let $z\in \Sigma_1^1$ be a point disjoint from $\partial\Sigma_1^1$ and all the push-off curves homotopic to $a_g$ and $b_g$ which are used to construct the tori $H_i$, as shown in Figure \ref{fig:B_lag_tori}. Then $\Sigma:= \{z\}\times \Sigma_2\subset B$ may be regarded as the square zero section of $\pi_n$. Since it's disjoint from all the Luttinger surgeries and where the fiber sum is taken, it survives as a square zero embedded surface in $B_m$. So $\pi_1(B_m-\Sigma)$ is normally generated by $x_2, y_2,$ and $\mu_\Sigma$. By the same reasoning as before, we can further refine the normal generating set of $\pi_1(B_m-\Sigma)$ to $\{x_2,y_2,\partial(\Sigma_1^1)\times \{q\}\}$.

Next we perform Luttinger surgery on $C$, whose fundamental group is generated by 
$\{a_1,b_1,\ldots, a_{g-1},b_{g-1}\}$ and $\{x_2,y_2\}$. In this case, we can use the exact same argument as in \cite[Section 4.2]{baykurhamadaspin}, where they show that $C$ may be Luttinger-surgered into a manifold $C'$ whose fundamental group is normally generated by $a$ and $b$, where $a$ and $b$ are disjoint homologically independent curves of $\Sigma_{g-1}^1\times \{q\}$. We glue $C'$ to $B_m-\Sigma$ along the same gluing region as $B$ and $C$ are glued along in the decomposition of $\mathcal Z_n$. After gluing, the normal generators $x_2, y_2, \partial(\Sigma_1^1)\times \{q\}$ of $\pi_1(B_m - \Sigma)$ are identified with elements of $\pi_1(C')$, and so may be (normally) expressed in terms of $a$ and $b$. Hence, setting $\mathcal D = B_m \cup C'$, we've just shown that $\pi_1(\mathcal D)$ and $\pi_1(\mathcal D - \Sigma)$ are normally generated by $a$ and $b$. Further employing the argument in \cite[Section 4.2]{baykurhamadaspin}, we may glue $\mathcal D$ to $A$ in such a way that the relations coming from the vanishing cycles of the Lefschetz fibration on $A$ kill $a$ and $b$, giving a simply--connected $4$--manifold which we denote $Z_{m,n}$. Note that $\Sigma$ continues to survive as a square zero embedded surface in $Z_{m,n}$, whose meridian is trivial in the complement, now that $a$ and $b$ are null-homotopic. To summarize our progress so far, $Z_{m,n}$ is obtained by applying Luttinger surgery to the symplectic fiber sum of $\mathcal Z_n \#_{T^2} E(2m)$. From this bird's-eye view, it's easier to see that $Z_{m,n}$ is symplectic and spin; all of the operations we've performed preserve the symplectic structure. The fiber sum gluing map between $E(n)$ and $\mathcal Z_n$ was chosen to preserve the spin structure, and the Luttinger surgeries also preserve the spin structure by Proposition \ref{prop:dual_torus_spin}. The algebraic invariants of $Z_{m,n}$ are obtained from those of $E(2m)$ and $\mathcal Z_n$, the latter of which has the same algebraic invariants as $\#(2n+1)S^2\times S^2$. Hence \[
	e(Z_{m,n}) = 24m+4n+4, \sigma(Z_{m,n}) = -16m.
\] If we perform the $\Z_2$--construction to $Z_{m,n}$ along $\Sigma$, the Euler characteristic will increase by two and the signature will remain the same. It follows that the $\Z_2$--construction of $Z_{m,n}$ along $\Sigma$ is an irreducible copy of $R_{2m,2n+3}$.

\subsubsection{Concluding geography graph for $w_2$ type (ii), $b_2^+$ even.} Figure \ref{fig:plane1} shows the $(a,b)$ points on the geography plane with the irreducibles homeomorphic to $L_2 \# E(a) \# b S^2 \times S^2$. Since $\sigma \leq 0$ we restrict to the first quadrant. Due to being spin, $a$ must be even. In addition, $b_2^+$ even implies that $b$ is odd when $a > 0$ and $b$ is even when $a=0$.

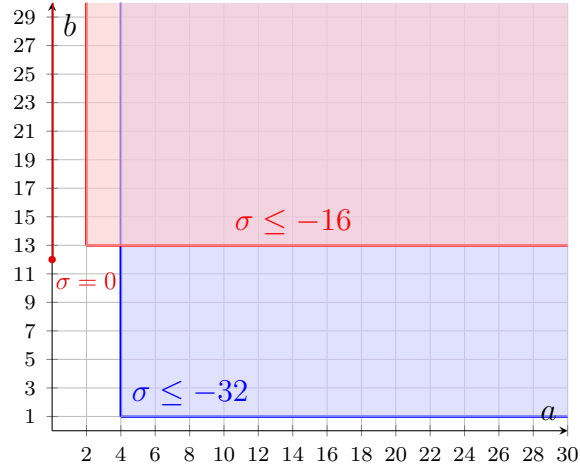
\begin{figure}[h]
	\centering
	\begin{tikzpicture}
		\begin{axis}[
			axis lines = middle,
			xlabel = $a$,
			ylabel = {$b$},
			xmin=0, xmax=30,
			ymin=0, ymax=30,
			xtick = {0,2,...,30},
			ytick = {0,1,3,...,29}, 
			grid=both,
			ticklabel style = {font=\tiny},
			every axis plot/.append style={thick}]

			\draw [blue,  line width=1pt] 
			(axis cs:4,1) -- (axis cs:30,1) 
			(axis cs:4,1) -- (axis cs:4,30);
			\fill[blue!20, opacity=0.6] 
			(axis cs:4,1) rectangle (axis cs:30,30);
			\node[blue, anchor=south west] at (axis cs:4,1) {$\sigma \leq -32$};
			
			\draw [red,  line width=1pt] 
			(axis cs:2,13) -- (axis cs:30,13) 
			(axis cs:2,13) -- (axis cs:2,30);
			\fill[red!20, opacity=0.6] 
			(axis cs:2,13) rectangle (axis cs:30,30);
			\node[red, anchor=south west] at (axis cs:10,13) {$\sigma \leq -16$};
			
			\draw [red!90!black, very thick] (axis cs:0,12) -- (axis cs:0,30); 
			\addplot [only marks, mark=*, mark size=1pt, red!90!black] coordinates {(0,12)};
			\node[red!90!black, anchor=north west, scale = 0.8] at (axis cs:-0.3,11.7) {$\sigma = 0$};

		\end{axis}
	\end{tikzpicture}
	\caption{Even $b_2^+$ and $w_2$-type $(ii)$}
	\label{fig:plane1}
\end{figure}

\subsection{Odd $\mathbf{b_2^+}$}\label{sec:type2_odd}

\subsubsection{$R_{2s+2,2n}$ for $s,n\geq 0$.}

Here, we will fill the geography region by applying $\Z_2$--construction along a torus to the family of manifolds $M_n(s)$ from the simply-connected geography of \cite{Park_Szabo}. If $s=0$, then $M_n(s)$ contains a copy of $E(2)_K - N_1$, where $N_1$ is a nucleus with symplectic $T_1$ and $S_1$. (See Remark \ref{rmk:Mn(s)_construction} for a reminder of this terminology.) Consider the symplectic torus $T_2$ inside its second nucleus and apply $\Z_2$--construction to $M_n(s)$ along $T_2$ to obtain $X_{n,s}$. The double cover of $X_{n,s}$ is simply-connected because the meridian of $T_2$ is trivial, since it bounds a punctured sphere $S_2$ inside the second nucleus $N_2$. If $s\geq 1$, then take a $\Z_2$--construction along a push-off of regular fiber in $E(2s)-F\subset M_n(s)$. Again, the meridian is homotopic to the meridian of $T_2$, which has a dual sphere, and so the resulting $X_{n,s}$ has a fundamental group of order two. The fact that $X_{n,s}$ is spin follows from Theorem \ref{thm:Z2_spin}. Since the Euler characteristic of a torus is zero, the invariants $e$ and $\sigma$ of $X_{n,s}$ and $M_n(s)$ coincide, i.e. $e(X_{n,s}) = 24s+4n+24$ and $\sigma(X_{n,s}) = -16s-16$. So $X_{n,s}$ is homeomorphic to $L_2 \# E(2s+2) \# 2n S^2 \times S^2$.

Consider the Lefschetz fibration $\pi\colon E(2m)\to S^2$ for $m\geq 1$, and let $F$ be a regular torus fiber. The following fact will often be used in our remaining constructions.

\begin{lemma} \label{lem:Em_lag_T2}
	There exists a homologically essential Lagrangian torus $\Lambda$ in $E(2m)$ which is disjoint from $F$, such that $\pi_1(E(2m)- (F \cup \Lambda)) = 1$.
\end{lemma}

\begin{proof}
	Let $\gamma$ be a curve on the base diagram for $E(2m)$ which is disjoint from $\pi(F)$ and encircles twelve critical values of $\pi$, such that the monodromy associated to $\gamma$ is $(t_at_b)^6$, where $a$ and $b$ are the symplectic generators for $H_1(T^2;\Z)$. Then $\pi^{-1}(\gamma)$ is an embedded copy of $S^1\times F'$, where $F'$ is another regular fiber away from $F$. Set $\Lambda = \gamma \times a \subset \pi^{-1}(\gamma)$, where $a$ is the symplectic $H_1$ generator on $F'$. This is a Lagrangian torus that's disjoint from $F$. To see that $\Lambda$ is homologically essential, observe that $\pi$ restricted to the interior of $\gamma$ is a Lefschetz fibration over $D^2$ with $a$ and $b$ as vanishing cycles. The curve $b$ on $F'\subset \pi^{-1}(\gamma)$ is a vanishing cycle, and so bounds a disk which will intersect $\Lambda$ transversely at a point. The same is true for $\pi$ restricted to the exterior of $\gamma$. The union of these two disks is a $(-2)$--sphere intersecting $\Lambda$ transversely at a point, showing that it is homologically essential and also that its complement is simply--connected. Moreover, this $(-2)$--sphere is disjoint from $F$ and the $(-2m)$--section of $\pi$ \cite[Lemma 3.1.10]{GompfStip}. Therefore $E(2m)-(F \cup \Lambda)$ is also simply--connected.
\end{proof}

\subsubsection{$R_{2m,2n+2}$ for $n\geq 5$, $m\geq 1$.} Recall the manifold $Z_{m,n}$ from Section \ref{sec:sigma_0_generalization}, which is obtained by applying Luttinger surgery to embedded tori in $\mathcal Z_n \#_{T^2} E(2m)$. The fiber sum is taken along a regular fiber $F$ in $E(2m)$. First note that for all $m\geq 1$, there exists a homologically essential Lagrangian torus $\Lambda\subset E(2m)$ which is disjoint from $F$. This torus is obtained exactly as in Lemma \ref{lem:Em_lag_T2}, by taking the parallel transport of an $H_1$ generator of $F$ along a curve on the base $S^2$ which bounds the sub-factorization $(t_at_b)^6$. Moreover $\Lambda$ has a simply--connected complement in $E(2m)$, since it has a geometrically dual $(-2)$--sphere formed from the two cores of the $2$--handles coming from the vanishing cycles of $E(2m)$. This $(-2)$--sphere is disjoint from $F$, and so $Z_{m,n}-\Lambda$ remains simply--connected. Hence, we may apply the $\Z_2$--construction to $Z_{m,n}$ along $\Lambda$. The resulting manifold will be irreducible and spin, with the same Euler characteristic and signature as $Z_{m,n}$. This gives an irreducible copy of $L_2\# E(2m) \# (2n+2) S^2\times S^2$ for $m\geq 1$, $n\geq 5$.

\subsubsection{$R_{0,2m+1}$ for $m\geq 6$.} To realize coordinates on the $\sigma = 0$ line, we use the manifold $Z_m$ constructed in \cite{baykurhamadaspin}, which is an irreducible copy of $\#_{2m+1} S^2\times S^2$ for $m\geq 5$. As described in \cite[Addendum 11]{baykurhamadaspin}, when $m\geq 6$, the manifold $Z_m$ has a homologically essential Lagrangian torus $T$ with a simply--connected complement. If we apply the $\Z_2$--construction to $Z_m$ along $T$, we obtain a spin manifold with order two $\pi_1$ and the same algebraic invariants as $Z_m$. Hence, we realize an irreducible copy of $L_2\#(2m+1) S^2\times S^2$ for $m\geq 6$.

\subsubsection{Concluding geography graph for $w_2$ type (ii), $b_2^+$ odd.} Figure \ref{fig:plane2} shows the $(a,b)$ points on the geography plane with the irreducibles homeomorphic to $L_2 \# E(a) \# b S^2 \times S^2$. Due to being spin, $a$ must be even. In addition, $b_2^+$ odd implies that $b$ is even when $a > 0$ and $b$ is odd when $a=0$.

\begin{figure}[h]
	\centering
	\begin{tikzpicture}
		\begin{axis}[
			axis lines = middle,
			xlabel = $a$,
			ylabel = {$b$},
			xmin=0, xmax=30,
			ymin=0, ymax=30,
			xtick = {0,2,...,30},
			ytick = {0,1,3,...,29}, 
			grid=both,
			ticklabel style = {font=\tiny},
			every axis plot/.append style={thick}]
			+
			
			\draw [red,  line width=1pt] 
			(axis cs:2,0) -- (axis cs:30,0) 
			(axis cs:2,0) -- (axis cs:2,30);
			\fill[red!20, opacity=0.6] 
			(axis cs:2,0) rectangle (axis cs:30,30);
			\node[red, anchor=south west] at (axis cs:10,13) {$\sigma \leq -16$};
			
			\draw [red!90!black, very thick] (axis cs:0,13) -- (axis cs:0,30); 
			\addplot [only marks, mark=*, mark size=1pt, red!90!black] coordinates {(0,13)};
			\node[red!90!black, anchor=north west, scale = 0.8] at (axis cs:-0.3,13.3) {$\sigma = 0$};

		\end{axis}
	\end{tikzpicture}
	\caption{Odd $b_2^+$ and $w_2$-type $(ii)$}
	\label{fig:plane2}
\end{figure}
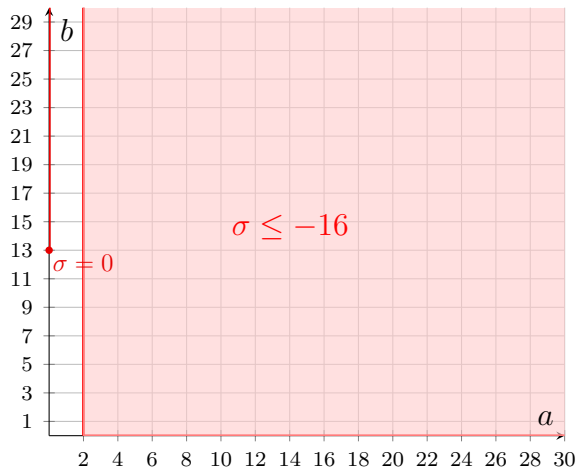

\section{$\Z_2$--geography for $w_2$--type (iii).}\label{sec:type3}

This section aims to construct irreducible $4$--manifolds with order two $\pi_1$ and $w_2$--type $(iii)$, meaning that they're non-spin but have spin universal covers. As discussed in Section \ref{sec:history}, we will populate the geography plane with manifolds having intersection forms isomorphic to $a(-E_8)\oplus bH$, where $b\geq 2a-1$. Whereas manifolds with $w_2$--type (ii) have signature divisible by 16, manifolds with $w_2$--type $(iii)$ have signature divisible by eight. So in this section, $a$ can be even or odd. In what follows, we will have separate constructions for when $\sigma\equiv 8\bmod 16$ and when $\sigma \equiv 0\bmod 16$. When $\sigma\equiv 8\bmod 16$, we can use Rokhlin's theorem to argue that the manifold is non-spin. On the other hand, detecting that a manifold with $\sigma \equiv 0\bmod 16$ is not spin is more delicate. Our strategy in this case is to find an embedded non-orientable surface of odd self-intersection. By Proposition \ref{prop:w_2_nonorientable_surface}, such an embedding obstructs a spin structure. We will also, as done previously, divide our work into constructions for $b_2^+$ even and $b_2^+$ odd. We will use the coordinate system given by the manifold's intersection form. That is, we will populate the $(a,b)$ plane where $Q_X = a(-E_8)\oplus bH$, and $b\geq 2a-1$.

\subsection{Even $\mathbf{b_2^+}$}\label{sec:type3_even} When populating the region of the geography plane with $b_2^+$ even, we will construct manifolds whose intersection forms are $a(-E_8)\oplus bH$, where $b$ is even.

\subsubsection{$(2s+5)(-E_8)\oplus (4s+2n+10) H$ for $n,s\geq 0$.}\label{sec:b_even_sigma_8}

\begin{figure}
	\centering
	\resizebox{!}{1.5in}{
	\begin{tikzpicture}

		\begin{scope}[]
			
			\begin{scope}[xshift = -.5cm]
				\draw[thick, color=blue, name path = b1] (-1.25,0) circle [radius = 1cm](-1.25,-1.1) node[anchor=north]{$c_2=y_1$};
				\draw[thick,color=blue, name path = a1] (-1.25,.5) .. controls +(.25,0) and +(.25,0) .. (-1.25,2)(-1.25,2) node [anchor = south] {$c_1=x_1$};
				\draw[thick,color=blue, dashed] (-1.25,.5) .. controls +(-.25,0) and +(-.25,0) .. (-1.25,2);
				
			\end{scope}
			
			\begin{scope}[xshift = .5cm]
				
				\draw[thick, color=blue, name path = b2] (1.25,0) circle [radius = 1cm] (1.25,-1.1) node[anchor=north] {$c_4=y_2$};
				
				\draw[thick,color=blue, name path = a1] (1.25,.5) .. controls +(.25,0) and +(.25,0) .. (1.25,2)(1.25,2) node [anchor = south] {$c_5=x_2$};
				\draw[thick,color=blue, dashed] (1.25,.5) .. controls +(-.25,0) and +(-.25,0) .. (1.25,2);
				
			\end{scope}
			
			\draw[thick,color=blue] (-1.25,0) .. controls +(0,-.5) and + (0,-.5) .. (1.25,0) (0,-.75) node {$c_3$};
			\draw[thick,color=blue,dashed] (-1.25,0) .. controls +(0,.5) and + (0,.5) .. (1.25,0);
			
			\filldraw[fill = pink!60!white,draw=red,thick] (3.5,0) ellipse (.25 and .5) (3.5,0) node[color=red] {$\delta$};

			\draw[very thick] (2,2) -- (-2,2) .. controls +(-1,0) and + (0,1) .. (-4,0).. controls +(0,-1) and +(-1,0) .. (-2,-2) -- (2,-2) .. controls +(1,0) and +(0,-1) .. (4,0) .. controls +(0,1) and +(1,0) .. (2,2);
			
			\draw[very thick] (-1.75,0) circle[radius = .5cm];
			\draw[very thick] (1.75	,0) circle[radius = .5cm];
			
		\end{scope}
		
	\end{tikzpicture}
	
	}
	\caption{}
	\label{fig:sigma2_type3}	
\end{figure}

Consider the maximal chain formed by $5$ simple closed curves $c_i$ on $\Sigma_2$. If we remove a disk at the  ``right end'' of $\Sigma_2$, we will have the following well-known mapping class relation:\[
	(t_{c_1}t_{c_2}t_{c_3}t_{c_4})^{10} = t_\delta,
\] where $\delta$ is the boundary component. Figure \ref{fig:sigma2_type3} shows this system of curves. This positive factorization corresponds to a closed genus two Lefschetz fibration $\pi:X\to S^2$ admitting a square $(-1)$--section. Let $F$ be a regular fiber of $\pi$. We claim that $X-\nu F$ admits a spin structure. To show this, we appeal to the correspondence between spin Lefschetz fibrations over disks and quadratic forms on the regular fiber summarized in Section \ref{sec:symp_background}. Let $\{x_1,y_1,x_2,y_2\}$ be the symplectic generators of $H_1(\Sigma_2;\Z)$, which we will also regard as $\pi_1$ generators by a slight abuse of notation. The quadratic form $q:H_1(\Sigma_2;\Z_2)\to \Z_2$ which sends $x_1,y_1,y_2$ to one and $x_2$ to zero evaluates to one on each of the vanishing cycles in the positive factorization for $\pi$. Hence, the criteria are met for the Lefschetz fibration $X-\nu F\to D^2$ to admit a spin structure.

Next, we compute some algebraic invariants for $X-\nu F$. We first show that $\pi_1(X-\nu F)=1$. The positive factorization for $\pi$ immediately ensures that $x_1=y_1=y_2=1\in \pi_1(X-\nu F)$, so $\pi_1(X-\nu F)$ is cyclic, and hence isomorphic to $H_1(X-\nu F;\Z)$, which is generated by $x_2$. Then since $x_2$ is represented by $c_5$, and $c_1 \cup c_3 \cup c_5$ is null-homologous in $\Sigma_2$, we have that $[c_1]+[c_3]+[c_5]=0\in H_1(X-\nu F;\Z)$. Since the vanishing cycles in $\pi$ make $c_1$ and $c_3$ null-homologous, it follows that $c_5$ is also null-homologous, completing the proof that $\pi_1(X-\nu F)=1$. We compute the signature of $X$ using Endo's formula \cite{endo2000hyperelliptic}, which is applicable in this case because all genus two Lefschetz fibrations are hyperelliptic. The positive factorization for $X$ has $40$ non-separating vanishing cycles and 0 separating, and so $\sigma(X) = -\frac{1}{5}(3\cdot 40) = -24$. Lastly, the Euler characteristic for $X$ is $e(\Sigma_2)\cdot e(S^2)+40= 36$.

We proceed to perform the $\Z_2$--construction of $X$ along $F$. To this end, we take the double: $\tilde Z:= (X-\nu F)\cup_{a\times r} (X-\nu F)$, where $a\times r \in \Diff(S^1\times \Sigma_2)$ is the gluing map from Proposition \ref{prop:spin_preserving_map}, which preserves spin structures. Note that $\tilde Z$ is a fiber--reversing double of Lefschetz fibrations, where the gluing map reverses orientation of the fiber rather than the $S^1$ factor. See \cite[Section 3]{ArabadjiMorgan} for more details on this construction. The manifolds $\tilde Z$ admits a Lefschetz fibration with a section of square $-2$, which is obtained by gluing two copies of the $(-1)$--section of $X$ together in the fiber sum. Since the positive factorization of the fiber sum supports a spin structure, it follows that the total space $\tilde Z$ is spin. Note that is still applicable because $\tilde{Z}$ is a chiral Lefschetz fibration by \cite[Section 3]{ArabadjiMorgan}. 

Let $Z$ be the $\Z_2$--construction of $X$ along $F$. Then $Z$ has invariants $e(Z)=38$ and $\sigma(Z)=-24$. Its fundamental group is $\Z_2$, with $\tilde Z$ as its universal cover. Since its signature is not divisible by 16, $Z$ can't be spin. But we just saw that its universal cover $\tilde Z$ is spin, so $Z$ must have $w_2$--type $(iii)$. Hence, $Z$ is an irreducible manifold with intersection form $Q_Z = 3(-E_8) \oplus 6H$.

Next we will use symplectic fiber sums to generalize the above construction to obtain other type $(iii)$ points with $\sigma \equiv 8\bmod 16$.

\begin{lemma}\label{lem:F_lambda_disjoint}
	There exists a torus $\Lambda\subset X$ and a genus two surface $F$ which are disjoint, symplectic, and have trivial normal bundles, such that $\pi_1(X-(\nu F\cup \nu \Lambda))=1$.
\end{lemma}

\begin{proof}
	Let $\pi\colon X\to S^2$ be the Lefschetz fibration described above with regular fiber $F$. The monodromy factorization for the Lefschetz fibration on $X$ is $\mu = (t_{c_1}t_{c_2}t_{c_3}t_{c_4})^{10}$. This monodromy is Hurwitz equivalent to $\mu' = (t_{c_1}t_{c_2})^6 W (t_{c_1}t_{c_2}t_{c_3}t_{c_4})^4$, where $W$ is a word comprised of six conjugates of $t_{c_3}t_{c_4}$. Hence, we can consider $\mu'$ to be the positive factorization associated to $\pi$. Let $\gamma$ be a simple closed curve on the base $S^2$ encircling the critical values of sub-factorization $(t_{c_1}t_{c_2})^6$ in $\mu'$. Since $(t_{c_1}t_{c_2})^6$ fixes $c_2$, we may take the parallel transport of $c_2$ along $\gamma$ to get a Lagrangian torus $\Lambda$. Similar to Lemma \ref{lem:Em_lag_T2}, this torus is disjoint from $F$ and is homologically essential. To see the latter, we observe that because $c_1$ shows up in the monodromy $\mu'$ on {\it both sides} of the curve $\gamma$, we may form a $(-2)$--sphere out of two $(-1)$--disks bound by $c_1$ which will intersect $\Lambda$ transversely once. Since $\Lambda$ is homologically essential, we may perturb the symplectic form to make it symplectic. The $(-2)$--sphere also provides a null-homotopy for the meridian of $\Lambda$ in $(X-\nu F) - \nu \Lambda$, showing that $\pi_1(X-(\nu \Lambda\cup \nu F))\cong \pi_1(X-\nu F)=1$.
\end{proof}

With this established, we can generalize the above construction by first taking a symplectic fiber sum of $X$ with other simply--connected, spin manifolds along $\Lambda$, and then performing the $\Z_2$--construction to this fiber sum along $F$. Recall $M_n(s)$ from Section \ref{sec:sigma<-32}. One can easily cook up a symplectic torus $T$ in $M_n(s)$ with a trivial normal bundle, e.g. pick a parallel copy of $T_2$ in the second nucleus of $E(2)_K$. Then we take a symplectic fiber sum $X_n(s) := M_n(s) \#_{T=\Lambda} X$. Then $X_n(s)-F$ is still simply-connected by the Seifert-van Kampen theorem. Apply $\Z_2$--construction to obtain $X'_n(s)$ with a double cover $\tilde{X}_n(s)$. The latter is a symplectic and spin because it is a fiber sum of $\tilde{Z}$ and two copies of $M_n(s)$, all of which are symplectic and spin.

Then $X_n(s)$ has invariants $e(X_n(s))= 24s+4n+60$ and $\sigma(X_n(s)) = -16s-40$ and so $X'_n(s)$ has invariants $e(X'_n(s))= 24s+4n+62$ and $\sigma(X'_n(s)) = -16s-40$. Since $\sigma(X'_n(s))$ is not divisible by $16$, it can't be spin, which confirms that it has $w_2$--type $(iii)$. Hence, $Z_n(s)$ is an irreducible manifold with order two $\pi_1$ and intersection form isomorphic to $(2s+5)(-E_8)\oplus (4s+2n+10) H$, where $n,s\geq 0$.

\subsubsection{$3(-E_8)\oplus (2m+8)H$ for $m\geq 5$.} \label{sec:b_even_with_Zm}

We can modify the argument in Section \ref{sec:b_even_sigma_8} to reach coordinates on the $\sigma = -24$ line. Let $X$ be the same $4$--manifold as in Section \ref{sec:b_even_sigma_8}, with embedded surfaces $\Lambda$ and $F$. Consider the manifold $Z_m$ constructed in \cite{baykurhamadaspin}, which is an irreducible copy of $\#_{2m+1} S^2\times S^2$ for $m\geq 5$. As described in Section \ref{sec:sigma_0_generalization}, $Z_m$ is related to a Lefschetz fibration $\mathcal Z_m$ through Luttinger surgery on a chain of $k$ disjoint Lagrangian tori. Let $Z'_m$ be obtained from $\mathcal Z_m$ by performing the first $k-1$ surgeries performed to get $Z_m$, but not performing the last surgery. That way $Z'_m$ has a homologically essential Lagrangian torus $T$, and $Z_m$ is obtained by performing Luttinger surgery to $Z'_m$ along $T$. Then perform the $\Z_2$--construction to $Z'_m\#_{T=\Lambda} X$, again along the genus two surface $F\subset X-\Lambda$. As before, since $\pi_1(X-(F\cup \Lambda))$ is trivial, $Z'_m\#_{T=\Lambda} X - F$ is simply-connected and spin, so the $\Z_2$--construction will have order two $\pi_1$, but will not be spin because of Rokhlin's theorem. This construction reaches coordinates with intersection form given by $3(-E_8)\oplus (2m+8)H$ for $m\geq 5$. 

In the two previous subsections we covered $\sigma = 8$ mod $16$. Next, we will cover coordinates with $\sigma = 0 \bmod 16$.

\subsubsection{$(2s+12)(-E_8) \oplus (4s+2n+32)H$ for $n,s\geq 0$} \label{sec:even_b_0_16}

We will follow the previous subsection but instead of a length 4 chain relation on $\Sigma_2$ we will employ length 8 chain relation on $\Sigma_4$. Most of the arguments will be analogous, so we will be concise.

Start from  the positive factorization $(c_1 \cdots c_8)^{18} $ on $\Sigma_4$ admitting a $(-1)$-section. Call the manifold $X$. By similar arguments as in \ref{sec:b_even_sigma_8}, the complement of the fiber $F$ is simply-connected. The Euler characteristic of $X$ is $e(X)=e(S^2)e(\Sigma_4)+144 = 132$. By Endo's formula $\sigma(X) = 18\cdot 8 \cdot \frac{-5}{9}=-80$. Apply the $\Z_2$--construction to get $Z$ having $e = 138$ and $\sigma = -80$. We claim that $X-\nu F$ admits a spin structure. To show this, we again appeal to the correspondence between spin structures and quadratic forms. The quadratic form $q\colon H_1(\Sigma_2;\Z_2)\to \Z_2$ which sends $x_1,x_3,y_1,y_2,y_3, y_4$ to one and $x_2,x_4$ to zero evaluates to one on each of the vanishing cycles. Hence, the criteria are met for the Lefschetz fibration $X-\nu F\to D^2$ to admit a spin structure. So the double along the fiber is also spin. Note the $(-1)$-section on $X$ lifts to $(-2)$-section on the double and consequently descents to $\RP^2$ with self-intersection $-1$ in $Z$.  Hence, by Proposition \ref{prop:w_2_nonorientable_surface}, $Z$ is non-spin.

Next, we use the constructions from Park and Szabó again \cite{Park_Szabo}. Pick a torus $\Lambda \subset X$ again by applying Hurwitz moves to split $(c_1c_2)^6$ and take a parallel transport along $c_2$. By similar arguments as before, $X- (F\cup \Lambda)$ is simply-connected. Let $M'_n(s)$ denote the $\Z_2$--construction of $M_n(s) \#_{T=\Lambda} X$ along $F$. Since the $(-1)$--section of $X$ is disjoint from $\Lambda$, $Z'_n(s)$ also has an embedded $\RP^2$ with self-intersection $(-1)$, and so is non-spin. To see that its double cover is spin, note that $M_n(s) \#_{T=\Lambda} (X-F)$ is spin. The invariants of $M'_n(s)$ are $e = 24s+4n+162$ and $\sigma = -16s-96$. Hence, $M'_n(s)$ is an irreducible manifold with order two $\pi_1$ and intersection form isomorphic to $(2s+12)(-E_8) \oplus (4s+2n+32)H$. Here, $n,s\geq 0$. 

\subsubsection{$10(-E_8) \oplus (2m+28) H$ for $m\geq 5$.}\label{sec:sigma_80}

This is similar to what we do in Section \ref{sec:b_even_with_Zm}. Let $X$ be the same as in Section \ref{sec:even_b_0_16}, with embedded surfaces $\Lambda$ and $F$. We perform the $\Z_2$--construction on $X\#_{\Lambda = T} Z_m$ along $F$. Since the $(-1)$--section in $X$ is disjoint from $\Lambda$, the $\Z_2$--construction will have an embedded copy of $\RP^2$ with odd self-intersection, so will be non-spin. Since $X\#_{\Lambda = T} Z_m - F$ is spin and simply-connected, the $\Z_2$--construction will have order two $\pi_1$ and a spin universal cover. This realizes coordinates with intersection forms given by $10 (-E_8) \oplus (2m+30) H$ for $m\geq 5$.

\subsubsection{Concluding geography graph for $w_2$ type (iii), $b_2^+$ even.} Figure \ref{fig:plane3} shows the $(a,b)$ points on the geography plane with the irreducibles having intersection forms of $a(-E_8)\oplus bH$. For this $w_2$ type, $a$ can be even or odd. The blue-shaded (resp. red-shaded) region indicates coordinates reached when $a$ is odd (resp. even), i.e. when $\sigma \equiv 8 \bmod 16$ (resp. $\sigma \equiv 0\bmod 16$). Since $b_2^+$ is even, $b$ is even.

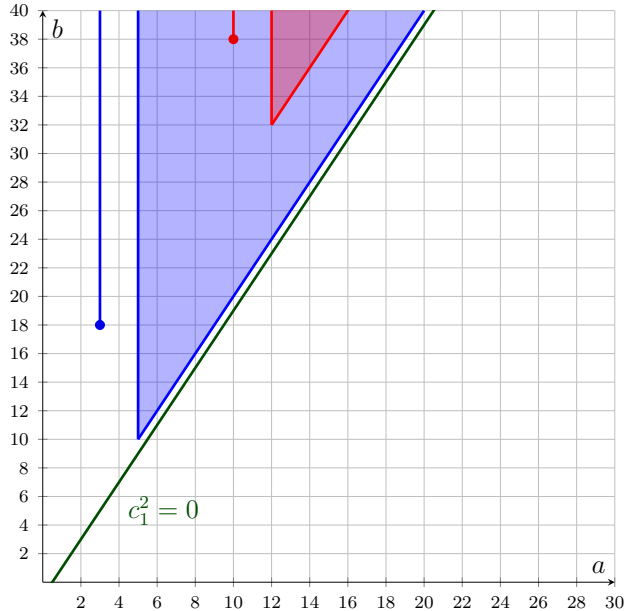
\begin{figure}[h]
		\centering
		\resizebox{!}{3.25in}{
			\begin{tikzpicture}
				\pgfplotsset{
					width= 300,
					height= 300
				}
				\begin{axis}[
					axis lines = middle,
					xlabel = {$a$},
					ylabel = {$b$},
					xmin=0, xmax=30,
					ymin=0, ymax=40,
					xtick = {0,2,...,30},
					ytick = {0,2,...,40}, 
					grid=both,
					ticklabel style = {font=\tiny}]
					
					
					\addplot[name path=slanted_ray, domain=5:20, draw=none] {2*x};
					\path[name path=vertical_ray] (axis cs:5,10) -- (axis cs:5,40);  
					\addplot[fill=blue, fill opacity=0.3] fill between[of=slanted_ray and vertical_ray];
					\addplot[domain=5:20, blue, very thick] {2*x};
					\draw[blue, very thick] (axis cs:5,10) -- (axis cs:5,40);
					\addplot [only marks, mark=*, mark size=2pt, blue!90!black] coordinates {(3,18)};
					\draw[blue, very thick] (axis cs:3,18) -- (axis cs:3,40);
					
					
					\addplot[name path=slanted_ray2, domain=12:20, draw=none] {2*x+8};
					\path[name path=vertical_ray2] (axis cs:12,32) -- (axis cs:12,40);  
					\addplot[fill=red, fill opacity=0.3] fill between[of=slanted_ray2 and vertical_ray2];
					\addplot[domain=12:27, red, very thick] {2*x+8};
					\draw[red, very thick] (axis cs:12,32) -- (axis cs:12,40);
					\draw[red, very thick] (axis cs:10,38) -- (axis cs:10,40);
					\addplot [only marks, mark=*, mark size=2pt, red!90!black] coordinates {(10,38)};

					\addplot[domain=0:22, green!30!black, very thick] {2*x -1};
					\node[green!30!black, right] at (axis cs:4, 5) {$c_1^2=0$};
				\end{axis}
			\end{tikzpicture}
			
		}
		\caption{Even $b_2^+$ and $w_2$-type $(iii)$. $\sigma = 8$ mod $16$ points shaded with blue, and $\sigma = 0$ mod $16$ with red. }
		\label{fig:plane3}
	\end{figure}

\subsection{Odd $\mathbf{b_2^+}$} \label{sec:type3_odd}

\subsubsection{$(2s+3)(-E_8) \oplus (4s+2n+5) H$ for $n,s\geq 0$.} The first idea is to apply the $\Z_2$--construction along a torus on a torus-fiber sum of $E(1)$ and $M_n(s)$ for $s, n\geq 0$. Here, $M_n(s)$ is again a simply-connected, spin 4-manifold from \cite{Park_Szabo} with the invariants $e = 24s+4n+24$ and $\sigma = -16s-16$. 

First, assume $s\geq 1$. By construction, it contains a copy of $E(2s)-F$, where $F$ is a regular fiber. Take a parallel copy $F_1\subset E(2s)$ of $F$ and fiber sum it with a regular fiber $F_2 \subset E(1)$ to obtain $U_{n,s} := M_n(s)\#_{F_1=F_2} E(1)$. Note that this is actually the same as replacing $E(2s)$ by $E(2s+1)$ in the construction of $M_n(s)$. The result is symplectic but not spin. However, it is spin outside a regular fiber $F_2^{||} \subset E(2s+1)$, since all elliptic fibrations are spin after removing a fiber. The next step is to apply the $\Z_2$--construction to $U_{n,s}$ along $F_2^{||}$. Let $\widetilde{U}_{n,s}$ be the double of $U_{n,s}$ along $F_2^{||}$, where the gluing is  given by the involution from Proposition \ref{prop:spin_preserving_map}. To see that $\widetilde U_{n,s}$ is simply--connected, note that the meridian of $F^{||}_2$ in $U_{n,s}-F^{||}_2$ is homotopic to the meridian of $F$ in $E(2s)$ which is, as per construction of $M_n(s)$, identified with the meridian of the torus in the nucleus of a knot surgered $K3$-surface. And that torus comes with a transverse sphere in the nucleus which makes the meridian null-homotopic. This makes $\pi_1(U'_{n,s})$ of order two. Next, $\widetilde{U}_{n,s}$ is spin in accordance with Proposition \ref{prop:spin_preserving_map}. Since $U'_{n,s}$ is not spin by the Rokhlin's theorem, it must have $w_2$--type $(iii)$.

Next, assume $s = 0$. Then $M_n(0)$ contains a copy of $K3-N_1$ (c.f. Remark \ref{rmk:Mn(s)_construction}). Let $T_2$ be a symplectic torus in the second nucleus of $K3$ with a transverse sphere $S_2$. Then let $U_{n,0}$ be the fiber sum of $M_n(0)$ and $E(1)$ along $T_2$ and a regular fiber $F$ on $E(1)$. This is again symplectic, simply-connected, and spin in the complement of a parallel copy $F^{\vert\vert }$ of $F$. The meridian of $F^{\vert\vert }$ bounds a disk since it is homotopic to the meridian of $T_2$, which has a dual sphere. Let $U'_{n,0}$ be the result of the $\Z_2$--construction on $U_{n,0}$ along $F^{\vert\vert }$. The same arguments as above show that $U'_{n,0}$ has $w_2$ type $(iii)$ and order two $\pi_1$.

For $n,s\geq 0$, the invariants of $U_{n,s}$ are $e = 24s+4n+36$ and $\sigma = -16s-24$. So are the invariants of $U'_{n,s}$. Since their signatures are congruent to $8\bmod 16$, $U'_{n,s}$ are of $w_2$-type $(iii)$. Finally, the intersection form of $U'_{n,s}$ is isomorphic to $(2s+3)(-E_8) \oplus (4s+2n+5) H$.

\subsubsection{$n(-E_8)\oplus (2n+2m+1) H$ for $n\geq 3$ and odd, and $m\geq 5$.} This construction only covers coordinates that have already been reached by the previous section. We nonetheless include it to demonstrate another way of reaching this part of the geography, and it will serve as a starting point for our construction in the next section. Recall the manifold $Z'_m$ with a symplectically embedded torus $T$ from Section \ref{sec:b_even_with_Zm}. Consider the elliptic surface $E(n)$ for $n\geq 3$ and odd, which has a regular torus fiber $F$ and whose monodromy is $(t_xt_y)^{6n}$. Let $\Lambda\subset E(n)$ be a Lagrangian torus obtained in the usual way, by taking a parallel transport of the curve $x$ along a curve in the $S^2$ base which bounds a subfactorization $(t_xt_y)^6$ (c.f. Lemma \ref{lem:Em_lag_T2}). Note that because $n>1$, $\Lambda$ is homologically essential, so may be considered as a symplectic embedded surface after a perturbation. Consider the symplectic fiber sum $E(n) \#_{\Lambda = T} Z'_m$. This space is simply--connected because for $n>1$, $\Lambda$ has a geometrically dual $(-2)$--sphere formed from cores of the Lefschetz singularity $2$--handles. We may take the double of $E(n) \#_{\Lambda = T} Z'_m$ along $F$, gluing via a free orientation-reversing involution. The resulting manifold is a copy of $E(2n) \#_{T^2} 2Z_m$. Let $Y_{n,m}$ be the $\Z_2$--construction of $E(n) \#_{\Lambda = T} Z'_m$ along $F$. Since $E(n) \#_{\Lambda = T} Z'_m$ is simply--connected, $Y_{n,m}$ has order two $\pi_1$.  The algebraic invariants of $Y_{n,m}$ are $e(Y_{n,m}) = 12n+4m+4$ and $\sigma(Y_{n,m}) = -8n$. Since $n$ is odd, $\sigma(Y_{n,m})= 8\bmod 16$, and so $Y_{n,m}$ has $w_2$--type $(iii)$. Its intersection form is isomorphic to $n(-E_8)\oplus (2n+2m+1) H$. Here $n\geq 3$ and odd, and $m\geq 5$.

\subsubsection{$(-E_8) \oplus (3+2m) H$ for $m\geq 6$.} To realize coordinates on the $\sigma = -8$ line, we need to modify the above construction. As described in \cite[Addendum 11]{baykurhamadaspin}, when $m\geq 6$, the manifold $Z_m$ has a homologically essential torus $T$ with a simply--connected complement. Let $\Lambda$ be the Lagrangian torus in $E(1)$ which is disjoint from a regular fiber $F$. Note that in this case, we can't arrange for $\Lambda$ to be obtained via a parallel transport along a curve which bounds singularities on {\it both sides} because in this case $\Lambda$ is null-homologous. We can nonetheless take the (non-symplectic) fiber sum $Z_m\#_{T=\Lambda} E(1)$, where the particular gluing is not so important. The total space is simply--connected. Next we take the double of $Z_{m}\#_{T=\Lambda}E(1)$ along $F$, gluing by an orientation-reversing free involution. The resulting space is simply--connected, and in fact is a copy of $Z_m\#_{T=\Lambda_2} E(2) \#_{\Lambda_1 = T} Z_m$, where $\Lambda_1$ and $\Lambda_2$ are parallel copies of $\Lambda$ which survive in the fiber sum $E(2) = E(1)\#_{F} E(1)$. Now in this new fiber sum, $\Lambda$ is homologically essential, as the parallel transport curve bounds singularities on both sides. So even though $Z_{m}\#_{T=\Lambda}E(1)$ was not necessarily symplectic, the double $Z_m\#_{T=\Lambda_2} E(2) \#_{\Lambda_1 = T} Z_m$ does admit a symplectic structure. Moreover, the double is simply--connected and admits an orientation-preserving free involution. Quotienting by this involution realizes an irreducible manifold with order two $\pi_1$ and intersection form isomorphic to $(-E_8) \oplus (3+2m) H$. Here $m\geq 6$. 

\subsubsection{$(2s+6)(-E_8)\oplus (4s+2n+13) H$ for $n,s\geq 0$.}\label{sec:b_odd_sigma_0_16}

Consider the maximal chain formed by $7$ simple closed curves $c_i$ on $\Sigma_3$. By the chain relation we have $(t_{c_1}t_{c_2}t_{c_3}t_{c_4}t_{c_5}t_{c_6}t_{c_7})^8=t_{\delta_1}t_{\delta_2}$ where $\delta_1$ and $\delta_2$ are the boundaries of the surface obtained by taking closed regular neighborhood of union $c_i$. This relation gives us a Lefschetz fibration $f\colon X\to S^2$ with two $(-1)$-sections, $\sigma(X) = 7\cdot 8\cdot \frac{-4}{7}=-32$ and $e(X) = 48$. Also, if $F$ is a regular fiber of $X$, then $X-\nu(F)$ is spin and $\pi_1(X-\nu F)=1$. To see the first claim, note that the quadratic form $q:H_1(\Sigma_3,\Z_2)\to \Z_2$, which evaluates to $1$ on $x_1,x_3,y_1,y_2,y_3$ and evaluates to $0$ on $x_2$, evaluates to one on all curves in the $7$-chain in $\Sigma_3$.

Let $Z$ be the manifold obtained by applying the $\Z_2$--construction to $X$ along $F$. Similar to the argument in Section \ref{sec:even_b_0_16} for $b_2^+$ even and $\sigma \equiv 0\bmod 16$, there exists an $\RP^2$ with odd self-intersection in $Z$. The argument in Lemma \ref{lem:F_lambda_disjoint} applies to $X$ as well, showing that there exists a homologically essential Lagrangian torus $\Lambda \subset X$ that's disjoint from $F$. As before $X-(F\cup \Lambda)$ is simply-connected and the $(-1)$--section is disjoint from $\Lambda$. Thus, if we let $Z_n(s)$ be the $\Z_2$--construction to $M_n(s) \#_{T^2=\Lambda} X$ along $F$, then $Z_n(s)$ is irreducible, $\pi_1(Z_n(s))=\Z_2$ and has $w_2$-type $(iii)$. The invariants of $Z_n(s)$ are $\sigma(Z_n(s))=-16s-48$ and $e(Z_n(s))=24s+4n+76$. So its intersection form is $(2s+6)(-E_8)\oplus (4s+2n+13) H$.

\subsubsection{$(k+1)(-E_8)\oplus (2k+1)H$ for $k\geq 1$.}\label{sec:c=0}

Let $\pi_g\colon X_g\to S^2$ be the genus $g$ Lefschetz fibration admitted by $X_g=\CPsum{}{(4g+5)}$ for $g\geq 2$, whose positive factorization is $h^2$, where $h$ is the product of $4g+2$ Dehn twists which gives the hyperelliptic involution. Tanaka has shown that $\pi_g$ admits $4g+4$ distinct $(-1)$-sections \cite{tanaka}. Assume $g$ is odd, so $g=2k+1$ for $k\geq 1$. We will perform the $\Z_2$--construction to $X_g$ along a regular fiber $F_g$. Let $\widetilde Z_g$ be the double realized along the way during the $\Z_2$--construction, which is a copy of $(X_g-F_g)\cup_{a\times r} (X_g-F_g)$, with $a$ the antipodal map on $S^1$ and $r$ a reflection map on $F_g$ which fixes each vanishing cycle in the monodromy for $\pi_g$ up to isotopy. As described in \cite[Section 3]{ArabadjiMorgan}, $\widetilde Z_g$ admits a genus $g$ Lefschetz fibration whose positive factorization is $h^2rh^{-2}r^{-1}$, which is equal to $h^4$. Therefore the space $\widetilde Z_g$ is the same as the ordinary untwisted fiber sum of $X_g$ with itself, which is $E(g+1)=E(2k+2)$ \cite[Section 3]{BKS}. Hence, if $Z_g$ is the end product of the $\Z_2$--construction of $X_g$, its double cover $\widetilde Z_g$ is spin. But the $(-1)$-sections of $X_g$ prevent $Z_g$ from also being spin; to see this, just isotope one of the $(-1)$-sections so that it intersects a fixed point of the reflection map $r$ along $F_g$. Then that $(-1)$-section will descend to a copy of $\RP^2$ of odd square in $Z_g$, showing that $Z_g$ is not spin. Hence, $Z_g$ has $w_2$-type $(iii)$. Its algebraic invariants $e$ and $\sigma$ are given by $e(E(2k+2))/2$ and $\sigma(E(2k+2))/2$, respectively. So $Z_g$ has an intersection form of $(k+1)(-E_8)\oplus (2k+1)H$ for $k\geq 1$.

\subsubsection{$2(-E_8)\oplus (5+2m)H$ and $4(-E_8)\oplus (9+2m)H$ for $m\geq 5$.}\label{sec:c=0_extension}

We may take a fiber sum of $X_3$ from the previous section with a copy of $Z'_m$ from Section \ref{sec:b_even_with_Zm} to realize coordinates on the $\sigma = -16$ line. Recall that $Z'_m$ has a homologically essential Lagrangian torus $T$. Present $X_3$ as a union $X\cup X$ of a Lefschetz fibration $X$ over a disk with monodromy $h$. There's also a homologically essential Lagrangian torus $\Lambda\subset X_3$ with a simply--connected complement. This torus is obtained in the usual way, by taking a parallel transport of $c_1$ along the boundary of $X$. There is a dual sphere in $X_3$ which is a union of two $(-1)$-disks, one in each copy of $X$, with boundaries given by the vanishing cycle $c_2$. This means that $\Lambda$ is essential and has a simply-connected complement. Thus, $X_3 \#_{\Lambda = T} Z_m$ is also simply-connected. The complement of $F_3$ in $X_3 \#_{\Lambda = T} Z_m$ is still simply-connected because a section of $X_3$ survives in $X_3 - \Lambda$. Then we perform the $\Z_2$--construction along a regular fiber $F_3$ in $X_3$.  Let $A_{3,m}$ be a result of the $\Z_2$--construction and let $\widetilde{A}_{3,m}$ be its double cover. This double cover is a copy of a fiber sum of $E(16)$ and two copies of $Z'_m$. Hence, the double cover is spin. On the other hand, the $(-1)$-section survives in $A_{3,m}$ as $\RP^2$ with an odd self-intersection, making $A_{3,m}$ non-spin. This realizes coordinates whose intersection forms are $2(-E_8)\oplus (2m+5)H$ for $m\geq 5$. We may do the same thing with $X_7$ instead of $X_3$ and realize the coordinates $4(-E_8)\oplus (2m+9)H$ for $m\geq 5$.

\begin{figure}[h!]
	\centering
	\resizebox{!}{3.25in}{
		\begin{tikzpicture}
			\pgfplotsset{
				width= 300,
				height= 300
			}
			\begin{axis}[
				axis lines = middle,
				xlabel = {$a$},
				ylabel = {$b$},
				xmin=0, xmax=30,
				ymin=0, ymax=40,
				xtick = {0,2,...,30},
				ytick = {0,2,...,40}, 
				grid=both,
				ticklabel style = {font=\tiny}]
				
				
				\addplot[name path=slanted_ray, domain=3:21, draw=none] {2*x-1};
				\path[name path=vertical_ray] (axis cs:3,5) -- (axis cs:3,40);  
				\addplot[fill=blue, fill opacity=0.3] fill between[of=slanted_ray and vertical_ray];
				\addplot[domain=3:20, blue, very thick] {2*x-1};
				\draw[blue, very thick] (axis cs:3,5) -- (axis cs:3,40);
				\addplot [only marks, mark=*, mark size=2pt, blue!90!black] coordinates {(1,15)};
				\draw[blue, very thick] (axis cs:1,15) -- (axis cs:1,40);
				
				
				\addplot[name path=slanted_ray2, domain=6:21, draw=none] {2*x+1};
				\path[name path=vertical_ray2] (axis cs:6,13) -- (axis cs:6,40);  
				\addplot[fill=red, fill opacity=0.3] fill between[of=slanted_ray2 and vertical_ray2];
				\addplot[domain=6:27, red, very thick] {2*x+1};
				\draw[red, very thick] (axis cs:6,13) -- (axis cs:6,40);
				\draw[red, very thick] (axis cs:4,19) -- (axis cs:4,40);
				\addplot [only marks, mark=*, mark size=2pt, red!90!black] coordinates {(4,19)};
				\draw[red, very thick] (axis cs:6,13) -- (axis cs:6,40);
				\draw[red, very thick] (axis cs:2,15) -- (axis cs:2,40);
				\addplot [only marks, mark=*, mark size=2pt, red!90!black] coordinates {(2,15)};

				\addplot[domain=0:22, green!30!black, very thick] {2*x -1};
				\node[green!30!black, right] at (axis cs:4, 5) {$c_1^2=0$};
				
				\draw[red, very thick] (axis cs:2,3) -- (axis cs:20.5,40);
				\addplot [only marks, mark=*, mark size=2pt, red!90!black] coordinates {(2,3)};
				
			\end{axis}
		\end{tikzpicture}
	}
	\caption{Odd $b_2^+$ and $w_2$-type $(iii)$. $\sigma \equiv 8$ mod $16$ points shaded with blue, and $\sigma \equiv 0$ mod $16$ with red.}
	\label{fig:plane4}
\end{figure}
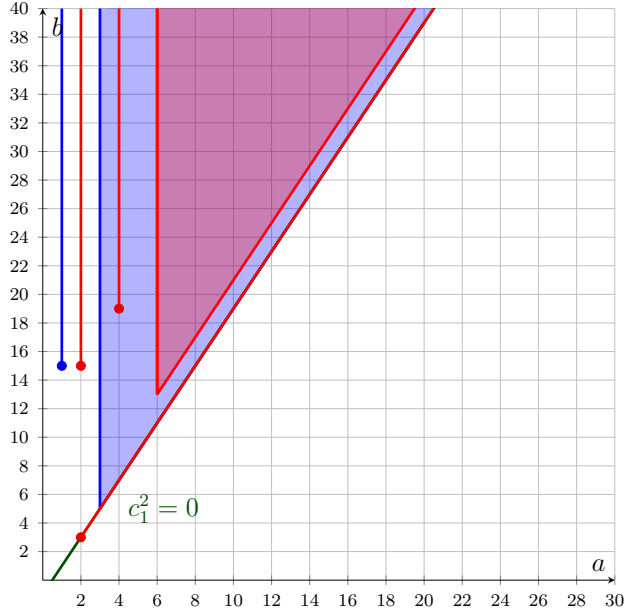

\subsubsection{Concluding geography graph for $w_2$-type (iii), $b_2^+$ odd.} Figure \ref{fig:plane4} shows the $(a,b)$ points on the geography plane with the irreducibles having intersection forms of $a(-E_8)\oplus bH$. The blue-shaded (resp. red-shaded) region indicates coordinates reached when $a$ is odd (resp. even), i.e. when $\sigma \equiv 8 \bmod 16$ (resp. $\sigma \equiv 0\bmod 16$). Since $b_2^+$ is odd, $b$ is odd.

\bibliography{refs} 

\bibliographystyle{alpha} 

\end{document}